\documentclass[11pt]{article}
\usepackage{amsmath,amsthm,amssymb,amscd}
\usepackage{ amssymb }
\usepackage{tikz}
\usepackage[all]{xy}
\usepackage[german,frenchb,english]{babel}
\usepackage{amsmath}
\usepackage{amssymb}
\usepackage{setspace}
\usepackage[lmargin=3.5cm,rmargin=3.2cm]{geometry}
\usepackage[backref]{hyperref}
\usepackage{ textcomp }

\usepackage{graphicx} 
\usepackage{titlepic} 

\DeclareMathOperator{\rnk}{rk}

\DeclareMathOperator{\codim}{codim}
\DeclareMathOperator{\sm}{sm}
\DeclareMathOperator{\Cc}{\mathcal{C}}

\DeclareMathOperator{\Lb}{\bar{\mathcal{L}}}
\DeclareMathOperator{\pr}{pr}

\DeclareMathOperator{\coker}{coker}
\DeclareMathOperator{\coh}{coh}
\DeclareMathOperator{\dR}{dR}

\DeclareMathOperator{\Dol}{Dol}

\DeclareMathOperator{\dual}{\vee}

\DeclareMathOperator{\Perf}{Perf}

\DeclareMathOperator{\Set}{Set}

\DeclareMathOperator{\QCoh}{QCoh}

\DeclareMathOperator{\supp}{supp}
\DeclareMathOperator{\Bun}{Bun}
\DeclareMathOperator{\I}{\mathcal{I}}
\DeclareMathOperator{\Qq}{\mathcal{Q}}
\DeclareMathOperator{\Qb}{\bar{\Qq}}
\DeclareMathOperator{\Jb}{\bar{J}}
\DeclareMathOperator{\Picb}{\overline{\Pic}}
\DeclareMathOperator{\J}{J}
\DeclareMathOperator{\Wb}{\mathbb{W}}
\DeclareMathOperator{\Db}{\mathbb{D}}
\DeclareMathOperator{\Hb}{\mathbb{H}}

\DeclareMathOperator{\Pp}{\mathcal{P}}
\DeclareMathOperator{\Hh}{\mathcal{H}}
\DeclareMathOperator{\Ll}{\mathcal{L}}
\DeclareMathOperator{\Pb}{\bar{\mathcal{P}}}

\DeclareMathOperator{\Pic}{\mathcal{P}ic}

\DeclareMathOperator{\Split}{\mathcal{S}}
\DeclareMathOperator{\Ssplit}{\mathcal{S}}

\DeclareMathOperator{\id}{id}

\DeclareMathOperator{\Coh}{Coh}

\DeclareMathOperator{\Sch}{Sch}

\DeclareMathOperator{\OmegaX}{\Omega^1_X}
\DeclareMathOperator{\OmegaXone}{\Omega^1_{X^{(1)}}}

\DeclareMathOperator{\Mod}{Mod}
\DeclareMathOperator{\res}{res}
\DeclareMathOperator{\invlim}{\varprojlim}
\DeclareMathOperator{\rk}{rk}
\DeclareMathOperator{\F}{\mathcal{F}}

\DeclareMathOperator{\Map}{Map}
\DeclareMathOperator{\Higgs}{\mathcal{M}_{Dol}}

\DeclareMathOperator{\intHiggs}{\mathcal{M}_{Dol}^{int}}
\DeclareMathOperator{\intSplit}{\mathcal{S}_{int}^0}
\DeclareMathOperator{\smHiggs}{\mathcal{M}_{Dol}^{sm}}
\DeclareMathOperator{\intLoc}{\mathcal{M}_{dR}^{int}}
\DeclareMathOperator{\ssLoc}{\mathcal{M}_{dR}^{ss}}
\DeclareMathOperator{\cLoc}{\mathcal{M}_{dR}^{coarse}}
\DeclareMathOperator{\sLoc}{\mathcal{M}_{dR}^{s}}
\DeclareMathOperator{\ssHiggs}{\mathcal{M}_{Dol}^{ss}}

\DeclareMathOperator{\G}{\mathbb{G}}

\DeclareMathOperator{\GL}{GL}
\DeclareMathOperator{\PGL}{PGL}
\DeclareMathOperator{\A}{\mathcal{A}}
\DeclareMathOperator{\smA}{\mathcal{A}^{(1)}_{sm}}

\DeclareMathOperator{\INTA}{\mathcal{A}_{int}}
\DeclareMathOperator{\intA}{\mathcal{A}^{(1)}_{int}}
\DeclareMathOperator{\M}{\mathcal{M}}

\DeclareMathOperator{\Y}{\mathcal{Y}}

\DeclareMathOperator{\Z}{\mathcal{Z}}
\DeclareMathOperator{\C}{\mathcal{C}}
\DeclareMathOperator{\D}{\mathcal{D}}
\DeclareMathOperator{\Ee}{\mathcal{E}}

\DeclareMathOperator{\Spec}{Spec}
\DeclareMathOperator{\Loc}{\mathcal{M}_{dR}}
\DeclareMathOperator{\HdR}{\chi_{dR}}
\DeclareMathOperator{\invHdR}{\chi^{-1}_{dR}}
\DeclareMathOperator{\HDol}{\chi_{Dol}}
\DeclareMathOperator{\invHDol}{\chi^{-1}_{Dol}}

\DeclareMathOperator{\Sym}{Sym}

\DeclareMathOperator{\Eend}{\underline{End}}
\DeclareMathOperator{\Hhom}{\underline{Hom}}

\DeclareMathOperator{\Oo}{\mathcal{O}}

\DeclareMathOperator{\Fr}{Fr}

\begin{document}

\title{Moduli of flat connections in positive characteristic}

\author{Michael Groechenig}

\newtheorem{defi}{Definition}[section]
\newtheorem{thm}[defi]{Theorem}
\newtheorem{prop}[defi]{Proposition}
\newtheorem{cor}[defi]{Corollary} 
\newtheorem{lemma}[defi]{Lemma}
\newtheorem{rmk}[defi]{Remark}
\newtheorem{cl}[defi]{Claim}
\newtheorem{ex}[defi]{Example}
\newtheorem{ass}[defi]{Assumption}
\newtheorem{conj}[defi]{Conjecture}

\maketitle

\begin{abstract}Exploiting the description of rings of differential operators as Azumaya algebras on cotangent bundles, we show that the moduli stack of flat connections on a curve (allowed to acquire orbifold points), defined over an algebraically closed field of positive characteristic is \'etale locally equivalent to a moduli stack of Higgs bundles over the Hitchin base. We then study the interplay with stability and generalise a result of Laszlo--Pauly, concerning properness of the Hitchin map. Using Arinkin's autoduality of compactified Jacobians we extend the main result of Bezrukavnikov--Braverman, the Langlands correspondence for $D$-modules in positive characteristic for smooth spectral curves, to the locus of integral spectral curves. We prove that Arinkin's autoduality satisfies an analogue of the Hecke eigenproperty.\end{abstract}

\tableofcontents

\section{Introduction}
Let $X$ be a smooth proper curve (which we allow to have orbifold points, in the sense of DM-stacks) defined over an algebraically closed field $k$, and $n$ a positive integer. We will be concerned with the study of various moduli stacks or spaces and their relation with each other. We use the notation $\Higgs$ for the moduli stack of rank $n$ Higgs bundles. Recall that a Higgs bundle is a pair $(E,\theta)$ consisting of a vector bundle $E$ on $X$ and a \emph{Higgs field} $\theta: E \to E \otimes \OmegaX.$
The moduli stack $\Higgs$ is equipped with a morphism $\chi_{\Dol}\colon \Higgs \to \A$, where $\A$ denotes the Hitchin base
$ \A = \bigoplus_{i=1}^n H^0(X,\Omega^{\otimes i}_X), $
and $\chi_{\Dol}$ sends the Higgs pair $(E,\theta)$ to the coefficients of the characteristic polynomial of $\theta$. The induced morphism for the moduli space of semistable Higgs bundles is proper and its generic fibre is an abelian variety.

We denote by $\Loc$ the moduli stack of pairs $(E,\nabla)$, where $E$ is a rank $n$ vector bundle on $X$ and $\nabla\colon E \to E \otimes \OmegaX$ is an algebraic flat connection.\footnote{We also use the term \emph{local system} for a pair $(E,\nabla)$.}

The aim of this paper is to study the moduli stack $\Loc$ for curves defined over an algebraically closed field $k$ of positive characteristic $p > 0$, which we fix once and for all. We will demonstrate that in this set-up the geometry of $\Loc$ is very similar to the geometry of $\Higgs$. 

Given a tangent vector field $\partial$ on a smooth scheme, the $p$-th iterate $\partial^p$ acts on functions as a derivation. The corresponding vector field is denoted by $\partial^{[p]}$. This allows us to define the \emph{p-curvature} of a flat connection in positive characteristic. It is given by the twisted endomorphism $E \to E \otimes \Fr^*\OmegaXone$ corresponding to
$ (\nabla_{\partial})^p - \nabla_{\partial^{[p]}}. $ The superscript ${}^{(1)}$ denotes the Frobenius twist of a variety, which the reader may safely ignore for now.
It has been shown in \cite{MR1810690} that the characteristic polynomial of the p-curvature gives rise to a natural morphism $\chi_{\dR}\colon \Loc(X) \to \A^{(1)}$ to the Hitchin base. We show that from the perspective of descent theory this map is indistinguishable from the Hitchin fibration $\Higgs(X^{(1)}) \to \A^{(1)}$.

\begin{thm}\label{thm:mainresult}
The $\A^{(1)}$-stacks $\Loc(X)$ and $\Higgs(X^{(1)})$ are \'etale locally equivalent\footnote{In this paper every covering in the \'etale topology is assumed to be fppf and \'etale.} over the Hitchin base $\A^{(1)}$. The same statement is true for the semistable loci $\ssLoc(X)$ and $\ssHiggs(X^{(1)})$.
\end{thm}

The proof of this Theorem relies on the techniques introduced in \cite{Bezrukavnikov:mz} and \cite{Bezrukavnikov:fr}. It has been shown that the ring of differential operators in positive characteristic can be described as an Azumaya algebra on the cotangent bundle. As Higgs bundles can be thought of as certain coherent sheaves on cotangent bundles, this allows us to interpret local systems as twisted versions of Higgs bundles. In \cite{Bezrukavnikov:fr} it has been shown that this twistedness carries over to the moduli stacks, if we restrict to the open substacks corresponding to smooth spectral curves. Our main result \ref{thm:mainresult} is a generalisation of this observation to all spectral curves, including highly singular and non-reduced ones. 

The results described above have appeared in the authors PhD thesis, and were first disseminated in a 2013 preprint. After the preprint appeared, a partial generalisation of Theorem \ref{thm:main} has been obtained by Chen--Zhu \cite{Chen:2013aa}. They prove that for a reductive group $G$ the moduli stack of $G$-local systems on $X$ is \'etale locally equivalent to the moduli stack of $G$-Higgs bundles on $X^{(1)}$, provided the characteristic of the base field $k$ is big enough. What remains interesting about the $GL_n$-case considered in this article, is that it is possible to study the interaction of these \'etale local equivalences and stability conditions; which is not addressed in \cite{Chen:2013aa}.

After establishing Theorem \ref{thm:mainresult} we give two applications. First we deduce properness of the Hitchin map introduced by Laszlo and Pauly (\cite[Prop. 5.1]{MR1810690}), using Nitsure's properness result for the Hitchin map for Higgs bundles (\cite[Thm. 6.1]{MR1085642}).

\begin{cor}\label{cor:proper}
The morphism $\chi_{\dR}\colon \ssLoc(X) \to  \A^{(1)}$ is universally closed. This implies in particular that the corresponding map from the GIT moduli space of local systems to $\A^{(1)}$ is proper.
\end{cor}

Let us note that it is a consequence of the main result of \cite{MR1810690} that $\chi_{\dR}^{-1}(0)$ is universally closed over the base field. Therefore Corollary \ref{cor:proper} extends this result from a single fibre to the entire fibration. Theorem \ref{thm:main} and Corollary \ref{cor:proper} will be treated in the more general setting, of $X$ being an orbicurve, that is, a smooth, proper, tame Deligne-Mumford stack over $k$ (containing an open dense substack which is a scheme).

From now on we assume that $X$ is a smooth complete curve, without orbifold points.
As a second application we extend the main result of \cite{Bezrukavnikov:fr}. We establish a derived equivalence between the derived category of certain $D$-modules on the moduli stack of bundles $\Bun(X)$ and the derived category of coherent sheaves on an open substack $\intLoc(X) \subset \Loc$. As expected, this equivalence respects the action of the Hecke and multiplication operators. Due to the limitations of what is currently known about autoduality of compactified Jacobians in positive characteristic, we need to assume that the characteristic of the base field satisfies the estimate $p > 2n^2(h-1) + 1$, where $h$ denotes the genus of $X$. We use the notation $\intLoc(X)$ and $\intHiggs(X^{(1)})$ for the moduli stacks of local systems respectively Higgs bundles of rank $n$ with integral spectral curve. The recent work of Melo--Rapagnetta--Viviani \cite{Melo:2013aa} allows one to improve the restriction on the characteristic to $p \geq n^2(h-1) + 1$, as observed on p.3 of \emph{loc. cit.}\footnote{Nonetheless we believed it to be important to retain our original estimate for the characteristic, since \emph{loc.cit.} is citing our Remark \ref{rmk:haiman} as part of their justification of their improved estimate.} 
\begin{cor}\label{cor:Langlands}
There exists a canonical equivalence of derived categories intertwining the Hecke operator with the multiplication operator associated to the universal bundle $\mathcal{E}$ on $\Loc(X) \times X$, 
$$ D^b_{\mathrm{coh}}(\intLoc(X),\Oo) \cong D^b_{\mathrm{coh}}(\intHiggs(X^{(1)}),\D_{\Bun}). $$
\end{cor}

Arinkin's autoduality of compactified Jacobians (cf. \cite{Arinkin:2010uq}) plays an essential role in our argument. We observe that this result still holds in large enough characteristic, and use it to deduce the existence of an integral kernel giving rise to the equivalence. We relate the Hecke eigenproperty of this kernel to an analogous property for Arinkin's Fourier-Mukai transform (Theorem \ref{thm:Hecke}), which has been stated in \cite[sect. 6.4]{Arinkin:2010uq}.

The Geometric Langlands Conjecture in positive characteristic has been studied by several authors since the seminal article \cite{Bezrukavnikov:fr}. We refer the reader to Travkin's \cite{Travkin:2011fk} for a treatment of Quantum Geometric Langlands in positive characteristic, and Chen--Zhu's \cite{chenzhu} article for a discussion of the case of semi-simple structure groups. As far as we know, the present work is the first extension of \cite{Bezrukavnikov:fr} to venture outside of the locus of smooth spectral curves or cameral covers.

\textit{Acknowledgements.} During my time of research on this topic I have profited from interesting discussions with Dmitry Arinkin, Pierre-Henri Chaudouard, Tam\'as Hausel, Jochen Heinloth, Yakov Kremnitzer, G\'erard Laumon, Christian Pauly and Rapha\"el Rouquier. I thank them for the valuable insights they have provided. H\'el\`ene Esnault encouraged me to update this paper, and thereby disseminate the results of my thesis. I thank her for her encouragement and interest. I thank Yun Hao for pointing out a mistake in a previous version of Theorem \ref{thm:bnrHiggs} and Proposition \ref{prop:local-bnr}. I also want to thank the EPSRC for generous financial support received under the grant EP/G027110/1 while the orginal version of this article was written, and EP/G06170X/1 during the revision process.

\section{$D$-modules in positive characteristic}\label{D}

Our understanding of flat connections in positive characteristic is based on a result of \cite{Bezrukavnikov:mz}, which describes the ring of differential operators in terms of an Azumaya algebra on the cotangent bundle. Before reviewing this theory in \ref{ssec:D} we recall the notion of an Azumaya algebra. 

\subsection{Azumaya algebras}\label{ssec:azumaya}

We refer the interested reader to chapter 4 in Milne's book \cite{MR559531} for a more detailed account of the theory of Azumaya algebras.

\begin{defi}\label{defi:azumaya}
\begin{enumerate}
\item[(a)] Let $U$ be a DM-stack and $\D$ a quasi-coherent sheaf of algebras on $U$. We say that $\D$ is a \emph{split Azumaya algebra} if there exists a locally free sheaf $E$ and an isomorphism of sheaves of algebras $\D \cong \Eend(E)$. The locally free sheaf $E$ is called a \emph{splitting} of $\D$.
\item[(b)] Let $U$ be a DM-stack and $\D$ a quasi-coherent sheaf of algebras on $U$. We say that $\D$ is an \emph{Azumaya algebra} on $U$ if there exists an \'etale covering $\{U_i \to U\}$, such that $\D|_{U_i}$ is a split Azumaya algebra for every $i \in I$.
\end{enumerate}
\end{defi}

The following remark follows directly from the definitions.

\begin{rmk}
Let $\D$ be a split Azumaya algebra with a splitting $E$ of rank $m$. Then $\D$ is a locally free sheaf of rank $m^2$. In particular it follows from faithfully flat descent theory that every Azumaya algebra $\D$ is locally free.
\end{rmk}

Given a splitting $E$ for an Azumaya algebra $\D$ on $U$ we can define a functor from quasi-coherent sheaves to $\D$-modules $\QCoh(U) \to \QCoh(U,\D),$ which sends a quasi-coherent sheaf $F$ to $E \otimes F$. There is also a functor $\delta_E\colon \QCoh(U,\D) \to \QCoh(U),$ which sends the $\D$-module $G$ to $\Hhom_{\D}(E,G)$. It is well-known that these two functors are mutually equivalences (\cite[Thm. 18.11 \& 18.24]{MR1653294}), in particular we obtain the following lemma, which gives the algebraic foundation for the proof of Proposition \ref{prop:local-bnr}.

\begin{lemma}[Morita theory]\label{lemma:morita}
Let $\D$ be an Azumaya algebra on a DM-stack $U$. Every splitting $E$ of $\D$ induces an equivalence of categories
$$\delta_E\colon \QCoh(U,\D) \cong \QCoh(U).$$
\end{lemma}

We can easily produce other splittings by twisting them with a line bundle $L$. This makes sense, since there is a natural isomorphism $\Eend(E) \cong \Eend(E \otimes L)$. We see that the set of splittings can be naturally endowed with the structure of a category: an arrow between two splittings $(\phi,E) \rightarrow (\psi,F)$ is a pair $(\gamma,L)$, where $L$ is a line bundle and $\gamma$ an isomorphism $E \rightarrow F \otimes L$, such that the diagram

\[
\xymatrix{ \D \ar[r]^{\phi} \ar[d]^{id_{\D}} & \Eend(E)\ar[d]^{(\gamma^{-1})^*\otimes \gamma} \\
\D \ar[r]^{\psi} & \Eend(F) }
\]

commutes. This category is actually a groupoid, the inverse of the morphism $(\gamma,L)$ being given by $(\gamma^{-1},L^{-1})$.

\begin{lemma}\label{lemma:split_line}
Let $E$, $F$ be two vector bundles over a DM-stack $X$, which are both splittings of an Azumaya algebra $\D$. Then there exists a morphism $(\gamma,L)$ between the two splittings.
\end{lemma}
\begin{proof}
This is an easy consequence of the fact that for a commutative ring $R$ the center of the matrix ring $M_n(R)$ is given by $R\cdot \id$. We know that $\Hhom_{\D}(F,-)$ is a categorical equivalence between the categories of $\D$-modules and the category of $\Oo_X$-modules (Lemma \ref{lemma:morita}). Since $E$ is another splitting, and \'etale locally both are equivalent (to the trivial rank $n$ locally free sheaf), we need to have that $L = \Hhom_{\D}(F,E)$ is \'etale locally free, thus it is a line bundle. On the other hand, $L \otimes_{\Oo_X} F \cong E$, therefore we get a morphism $(\gamma,L)$ between both splittings.
\end{proof}

\begin{cor}
The stack of splittings of an Azumaya algebra $\D$ has a natural structure of a $\G_m$-gerbe. It will be denoted by $\Y_{\D}$.
\end{cor}

In the following we will neglect $\G_m$ from the notation and simply refer to $\G_m$-gerbes as \emph{gerbes}. Gerbes are a higher notion of line bundles, in the sense of isomorphism classes of line bundles correspond to elements of  $H^1_{\text{Zar}}(X,\Oo_X^{\times})$ and isomorphism classes of gerbes to elements of $H^2_{\text{\'et}}(X, \Oo_X^{\times})$. 

\begin{prop}\label{prop:obstruction_splitting}
An Azumaya algebra $\D$ on $X$ splits if and only if the gerbe of splittings is neutral, i.e. if and only if the corresponding class in $H^2_{\text{\'et}}(X, \Oo_X^{\times})$ vanishes.
\end{prop}
\begin{proof}
This follows directly from the definition. A gerbe is neutral if its structure map admits a section. By definition, such a section corresponds to a global splitting of $\D$.
\end{proof}

This discussion is summarised by the following short exact sequence of group schemes.
$$1 \rightarrow \G_m \rightarrow \GL_n \rightarrow \PGL_n \rightarrow 1,$$
which gives rise to a long exact sequence of non-abelian cohomology sets
$$H^1_{\text{\'et}}(X,\GL_n) \rightarrow H^1_{\text{\'et}}(X,\PGL_n) \rightarrow H^2_{\text{\'et}}(X,\G_m).$$

\subsection{D-modules}\label{ssec:D}

In this section we recall foundational results on the ring of differential operators in positive characteristic. More detailed expositions can be found in \cite{Bezrukavnikov:fr} and \cite{Bezrukavnikov:mz}.

Let $U$ be a DM-stack over a field $k$ of characteristic $p > 0$. The \emph{Frobenius twist} $U^{(1)}$ of $U$ is given by the following cartesian diagram
\[
\xymatrix{ U^{(1)} \ar[r]\ar[d] & U \ar[d] \\
\Spec k \ar[r]^F & \Spec k, }
\]
where $F\colon k \to k$ denotes the Frobenius map $\lambda \mapsto \lambda^p$. There exists a canonical $k$-linear morphism $Fr_U\colon U \to U^{(1)}$, which is referred to as \emph{Frobenius morphism}. If $U$ is the zero set of a polynomial $$f(x_1,\dots x_n) = \sum a_{i_1\dots i_n}x^{i_1}\cdots x^{i_n},$$ then $U^{(1)}$ is the zero set of the polynomial $$\sum a_{i_1\dots i_n}^px^{i_1}\cdots x^{i_n},$$ and $Fr_U$ is given by $x \mapsto x^p$.

Given a smooth DM-stack $Y$ defined over $k$, the ring of (crystalline) differential operators $D_Y$ is defined to be the universal enveloping algebra $U\Theta_Y$ of the Lie algebroid of tangent vector fields $\Theta_Y$. By definition $Fr_Y$ is an affine morphism, therefore the study of $D_Y$-modules is equivalent to the study of $Fr_*D_Y$-modules. The theorem below states that relative to its centre $Fr_*D_Y$ is an Azumaya algebra (Definition \ref{defi:azumaya}).
For a variety $Z$ we denote by $\pi\colon T^*Z \rightarrow Z$ the canonical projection.

\begin{thm}[\cite{Bezrukavnikov:mz}]\label{thm:D-Azumaya}
The centre $Z(Fr_*D_Y)$ is canonically equivalent to the structure sheaf of the cotangent bundle $\iota\colon \pi_*\Oo_{T^*Y^{(1)}} = \Sym \Theta_{Y^{(1)}}$. Moreover relative to its centre, the ring $Fr_*D_Y$ is an Azumaya algebra $\D_Y$ of rank $p^{2\dim Y}$.
\end{thm}

The identification of the centre with the cotangent bundle is realised by the $p$-curvature. If $\partial$ is a vector field on $X$, its $p$-th-power $\partial^p$ acts again as a derivation on functions, according to the Leibniz rule. We denote this vector field by $\partial^{[p]}$; like every vector field is gives rise to an order $\leq 1$ differential operator. We can also define $\partial^p$ as an element of the ring $D$, it is an order $\leq p$ differential operator. The difference $\partial^p - \partial^{[p]} \in D_Y$ is a non-zero differential operators of order $\leq p$. It can be shown to be a central element (cf. \cite[Lemma 1.3.1]{Bezrukavnikov:mz}). This expression is $p$-linear in $\partial$ and hence gives rise to a morphism 
\begin{equation}\label{eqn:iota}\iota\colon \pi_*\Oo_{T^*Y^{(1)}} \to Z(Fr_*D_Y).\end{equation}

This theorem implies the existence of an equivalence of categories (see lemma below) $$D_Y-\Mod \cong \D_Y-\Mod.$$ Note that objects in the latter category are no longer sheaves on $Y$ but on $T^*Y^{(1)}$.

\begin{lemma}\label{lemma:D}
\begin{enumerate}
\item[(a)] Let $M$ be a $D_Y$-module on $Y$, then there exists a $\D_Y$-module $\mathcal{M}$ on $T^*Y^{(1)}$, which is characterised up to a unique isomorphism by the condition 
\begin{equation}\label{eqn:condition}
\Fr_*M \cong \pi_*\mathcal{M},
\end{equation}
where $\pi\colon T^*Y^{(1)} \to Y^{(1)}$ denotes the canonical projection. Vice versa, if $\mathcal{M}$ is a $\D_X$-module on $T^*Y^{(1)}$, then there exists a $D_Y$-module $M$, characterised up to a unique isomorphism by the condition (\ref{eqn:condition}).
\item[(b)] The correspondence described in (a) is functorial, that is, can be refined to an equivalence of categories $$\Lambda\colon \QCoh(Y,D_Y) \cong \QCoh(T^*Y^{(1)},\D_Y),$$ which sends $M$ to $\mathcal{M}$.
\end{enumerate}
\end{lemma}

\begin{proof}
If $f\colon U \to V$ is an affine morphism, and $\mathcal{B}$ is a quasi-coherent sheaf of algebras on $U$, then $f_*$ induces an equivalence of categories $$f_*\colon \mathcal{B}-\Mod \cong f_*\mathcal{B}-\Mod.$$ Applying this equivalence once to $\Fr\colon Y \to Y^{(1)}$ and $D$, and once to $\pi\colon T^*Y^{(1)} \to Y^{(1)}$ and $\D$, we obtain the following chain of equivalences which implies the statement above
$$\QCoh(Y,D_Y) \cong \QCoh(Y^{(1)},\Fr_*D_Y) = \QCoh(Y^{(1)},\pi_*\D_Y) \cong \QCoh(T^*Y^{(1)},\D_Y).$$
This concludes the proof.
\end{proof}

\section{Local systems in positive characteristic}\label{local}

In this chapter we develop the theory of local systems, in analogy with the theory of Higgs bundles. The first two subsections are literature reviews, from subsection \ref{local-bnr} on, we present our contributions.

\subsection{The theory of Higgs bundles as a blueprint}\label{bnr}

Let $X$ be a smooth DM-stack. A Higgs field $\theta$ on a vector bundle $E$ on $X$ is a section of $\Eend(E) \otimes \OmegaX$, satisfying $\theta \wedge \theta = 0$. Equivalently, $E$ is endowed with the structure of a $(\Sym \Theta_X)$-module, where $\Theta_X$ denotes the sheaf of tangent vectors. We associate to $\theta$ its characteristic polynomial 
$$ a(\lambda) = \det(\lambda - \theta), $$
with coefficients $$a_i \in H^0(X,\Omega_X^{\otimes n-i})$$ for $i = 0, \dots, n-1$. The corresponding spectral cover $Y_a$ is the closed substack of $T^*X$ defined by the equation
$$ \lambda^n + a_{n-1} \lambda^{n-1} \cdots + a_0 = 0, $$
where $\lambda$ denotes the tautological section of the pullback $\pi^*\OmegaX$ to the cotangent space $T^*X$. Repeating this construction for a family of Higgs bundles parametrised by the scheme $S$, we obtain a morphism $$a\colon S \to \A,$$ which will also be referred to as the characteristic polynomial of the family of Higgs bundles. The space of spectral covers is the functor $\A\colon \Sch \rightarrow \Set$, which sends a scheme $S$ to the set $\bigoplus_{i = 0}^{n-1}\Gamma(X \times S,(\Omega_{X\times S/S}^1)^{\otimes i})$. The corresponding spectral cover is a family of orbicurves $Y_a \to S$, fitting into a commutative diagram
\[
\xymatrix{
Y_a \ar[r]^{\pi} \ar[rd] & X \times S \ar[d]\\
& S 
}
\]
If $X$ is proper, we have 
$$ \A = \Spec \Sym \left(\bigoplus_{i=1}^{n} H^0(X,\Omega_X^{\otimes i})\right)^{\vee}. $$

The scheme $\A$ parametrises the universal family $\phi\colon Y \to \A$ of spectral covers and in particular gives rise to a finite morphism
$\pi\colon Y \to X \times \A. $
\'Etale locally on $X$ we can trivialise the sheaf of tangent vector fields $\Theta_X$ and see that $$\pi_*\Oo_Y = \bigoplus_{i=0}^{n-1}p_1^*\Theta_X^{\otimes i}$$ is locally free. In particular we obtain that $\pi$ and $\phi$ are flat morphisms. Moreover using \cite[Prop. 7.8.4]{MR0163911} we see that $\phi$ cohomologically flat in degree zero. We record this observation for later use.

\begin{lemma}\label{lemma:pfcf}
For $X$ a smooth orbicurve, the morphism $\phi\colon Y \to \A$ is proper, flat and cohomologically flat in degree zero. Moreover, if $X$ is a curve (that is, a scheme) its geometric fibres are locally planar curves (that is, local embedding dimension 2).
\end{lemma}

Given a quasi-coherent sheaf $L$ on a spectral cover $Y_a$, we obtain a coherent sheaf on $T^*X$, that is an $(\Sym \Theta_X)$-module on $X$. Therefore the push-forward $\pi_{*}L$ is naturally endowed with a Higgs field. The Higgs sheaf $\pi_{*}L$ is a Higgs bundle if and only if $\pi_*L$ is locally free. This construction induces a natural bijection between Higgs bundles and certain coherent sheaves on spectral covers. It is usually referred to as the BNR correspondence, after Beauville--Narasimhan--Ramanan (see \cite{MR998478} for the case of integral spectral curves and \cite[Lemma 6.8]{MR1320603} for the general case). The statement given below is slightly weaker than the one given in \emph{loc. cit.}, which relates Higgs bundles to so-called pure sheaves. The version stated here has the advantage of being easier to prove and is sufficient for our purposes.

\begin{thm}[BNR correspondence]\label{thm:bnrHiggs}
Let $S$ be an arbitrary scheme, and $X$ a smooth DM-stack. There is a natural embedding of $S$-families of rank $n$ Higgs bundles $(E,\theta)$ on $X$, with characteristic polynomial $a = \chi(\theta)$, and coherent sheaves on the spectral cover $Y_a/S$, such that $\pi_{a,*}L$ is a locally free sheaf of rank $n$.
\end{thm}

\begin{proof}
Let $\theta\colon E \to E \otimes p_1^*\OmegaX$ be a Higgs field on a locally free sheaf of rank $n$ on $X \times S$, with characteristic polynomial $a$. Equivalently we can think of $\theta$ as a map $p_1^*\Theta_X \otimes E \to E$, where $\Theta_X$ denotes the sheaf of tangent vector fields on the variety $X$. This is the case since both $\OmegaX$ and $\Theta_X$ are locally free sheaves, which are dual to each other. Giving such a map, satisfying the integrality condition $\theta \wedge \theta = 0$, on the other hand is equivalent to endowing $E$ with the structure of a $\Sym p_1^*\Theta_X$-module. The Cayley-Hamilton Theorem implies that $E$ is moreover a $\Sym p_1^*\Theta_X/(a)$-module, which agrees with $\pi_*\Oo_{Y_a}$, where $\pi\colon Y_a \to X$ denotes the natural projection. Since $\pi$ is an affine morphism, we have that $\pi_*$ induces an equivalence between the category of coherent sheaves $L$ on $Y_a$, satisfying that $\pi_*L$ is locally free of rank $n$, and the category of rank $n$ Higgs bundles.
\end{proof}

\subsection{Local systems}\label{local}

Let us introduce the moduli problem of \emph{local systems}.
\begin{defi}\label{defi:localsystems}
Let $X$ be a smooth Deligne-Mumford stack, $S$ a scheme. An $S$-family of \emph{local systems} on $X$ is a pair $(\Ee,\nabla)$, where $\Ee$ is a vector bundle on $X \times S$, and $\nabla\colon \Ee \to \Ee \otimes \Omega_{X\times S/S}^1$ is an $\Oo_S$-linear map of sheaves satisfying the Leibniz rule $\nabla(fs) = (d_{X\times S/S}f) s + f \nabla s$, where $f$ is a section of $\Oo_{XÊ\times S}$ and $s$ a section of $\Ee$. Moreover, $\nabla$ has to satisfy the integrality condition $0 = \nabla^2\colon E \to E \otimes \Omega_{X\times S/S}^2,$ that is, be a \emph{flat connection}.
\end{defi}

From now on, we let $X$ be a smooth orbicurve, unless stated otherwise. By orbicurve we refer to the following particular case of a DM-stack.
\begin{defi}\label{defi:orbicurve_new}
An orbicurve $X$ over an algebraically closed field $k$ is a proper, tame (see \cite{MR2427954}) DM-stack of dimension $1$, defined over $k$, which contains an open dense substack which is a scheme.
\end{defi}

We emphasise that it is not necessary to impose the flatness condition $\nabla^2 = 0$ on the connection $\nabla$ for smooth orbicurves case. For dimension reasons, the sheaf $\Omega_{X \times S/S}^2$ vanishes, hence every connection on a smooth orbicurve is automatically flat.

\begin{defi}\label{defi:localstack}
The stack of local systems on $X$ will be denoted by $\Loc(X)$. By definition, it is given by the functor, sending a scheme $S$ to the groupoid of $S$-families of local systems on $X$.
\end{defi}

It has been shown by Laszlo and Pauly in \cite[Cor. 3.1]{MR1810690} that the stack $\Loc(X)$ is algebraic, if $X$ is an algebraic curve. Nonetheless, algebraicity will also follow from the main result of this chapter, which relates stacks of local systems to stacks of Higgs bundles (Theorem \ref{thm:main}).
%

Local systems inherit a stability condition from the one of bundles. Let $\mu(E) = \frac{\deg E}{\rnk E}$ denote the Mumford slope. We say that $(E,\nabla)$ is semistable if for every subconnection $(F,\nabla)$ we have $\mu(F) \leq \mu(E)$.

In the article \cite{MR0086359}, it was shown by M. Atiyah, that the characteristic classes of a vector bundle defined over a compact K\"ahler manifold, can be related to the obstruction of the existence of analytic connections. The same arguments also cover the case of a vector bundles defined on projective varieties of zero characteristic. As a corollary (Theorem 10 in \emph{loc. cit.}) he reproves a theorem of Weil, which states that a vector bundle defined over a curve of zero characteristic carries an algebraic connection, if and only if the degrees of all indecomposable summands are zero. We conclude from this discussion, that a local system in zero characteristic, is always semistable, since we have seen that a vector bundle carrying an algebraic connection has to be of degree zero. 

The same argument does not apply in the context of positive characteristic geometry. Nonetheless, it can be shown that over a perfect infinite field of characteristic $p$, a connection exists on a vector bundle if and only if for every indecomposable summand the degree satisfies  $p|\deg E$ (\cite{MR2237516}). 
\begin{defi}\label{defi:localstable}
We denote the stack of semistable local systems by $\ssLoc(X)$, the stack of stable local systems by $\sLoc(X)$. They are obtained by rigidifying the corresponding substacks of semistable, respectively stable objects. Here rigidification refers to forming the quotient by the action of $B\G_m$.
\end{defi}
In general we will not restrict the degree of our bundles.
As indicated by the discussion above, local systems in positive characteristic will not be automatically semistable, like it is the case in zero characteristic.

The following proposition can be derived from the main results of this chapter (Theorem \ref{thm:main} and Lemma \ref{lemma:p}) in the case of Higgs bundles.

\begin{prop}\label{prop:Loc_open}
The stacks $\ssLoc(X)$ and $\sLoc(X)$ are open substacks of the rigidification $[\Loc(X)/B\G_m]$. In particular, they are algebraic.
\end{prop}

We have already mentioned a central difference between the theory of local systems in zero and positive characteristic, when discussing the admissible degrees a local system can have (see the discussion Subsection \ref{local}). Another significant difference lies in the existence of a Hitchin map for local systems, emphasizing the relation with the theory of Higgs bundles. It has been studied by Laszlo and Pauly in \cite{MR1810690} and is based on the notion of $p$-curvature for flat connections.

\begin{defi}\label{defi:p-curvature}
Let $E$ be a quasi-coherent sheaf on a smooth scheme $X$, endowed with a flat connection $\nabla$. The \emph{$p$-curvature} of $\nabla$ is given by $$\Psi_{\nabla}(\partial)\colon s \mapsto (\nabla_{\partial})^p s - \nabla_{\partial^{[p]}} s, $$ where $\partial$ is a tangent vector field of $X$, $\partial^{[p]}$ is the tangent vector field given by the $p$-th power of the derivation $\partial$ (see Subsection \ref{ssec:D}), and $s$ denotes a section of $E$.
\end{defi}

The definition of the map $\iota$ (see Equation \eqref{eqn:iota}) has been inspired by the notion of $p$-curvature. Since a $D_X$-module is precisely a quasi-coherent sheaf on $X$ with a flat connection, we see that $\Psi_{\nabla}(\partial)(s)$ is given by applying the element $\iota(\partial)$ to the section $s$. Using this interpretation of $p$-curvature, one obtains \cite[Prop. 5.2 ]{Katz:fk}.
\begin{cor}\label{cor:p-linear}
For a quasi-coherent sheaf $E$ on a smooth $k$-scheme $X$, with flat connection $\nabla$, the $p$-curvature is \emph{$p$-linear}, that is, subject to the relation 
$$\Psi_{\nabla}(f\partial_1 + \partial_2) = f^p\Psi_{\nabla}(\partial_1) + \Psi_{\nabla}(\partial_2),$$ where $f$ is a section of $\Oo_X$, and $\partial_1$ and $\partial_2$ are sections of $\Theta_X$. This allows us to view $\Psi_{\nabla}$ as a morphism $$E \to E\otimes Fr^*\Omega_X^1,$$ where $Fr\colon X \to X^{(1)}$ denotes the Frobenius morphism (see Subsection \ref{ssec:D}).
\end{cor}
While a priori, $p$-curvature measures the deviation from $\nabla$ being a map of \emph{restricted $p$-Lie algebras}, there is a second interpretation as capturing the obstruction for $E$ descending along the Frobenius morphism $Fr\colon X \to X^{(1)}$. In order to make this observation precise, we have to endow Frobenius pullbacks $Fr^*E'$ with a connection.

\begin{lemma}\label{lemma:canonical-connection}
Every pullback $Fr^*E'$ of a quasi-coherent sheaf $E'$ on $X^{(1)}$ can be canonically endowed with a connection $\nabla^{can}$, which is well-defined by the property that $$\nabla (Fr^*s) \equiv 0,$$ for every section pulled back along the Frobenius. The connection $\nabla^{can}$ will be called the \emph{canonical connection} on $Fr^*E$. This construction gives rise to a functor $$Fr^*_{\nabla}\QCoh(X^{(1)}) \to \QCoh(X,D_X).$$
\end{lemma}

\begin{proof}
We write $(F^*E,\nabla^{can}) = (F^{-1}E \otimes_{F^{-1}\Oo_{X^{(1)}}} \Oo_X, 1 \otimes d),$ which is a functorial definition of the canonical connection.\footnote{I thank H\'el\`ene Esnault for pointing out this simplification.}
\end{proof}

Now we can state the descent-theoretic interpretation of $p$-curvature. It is a result of Cartier, which Katz states as Theorem 5.1 in \emph{loc. cit.}. The proof given here relies on the interpretation of $D_X$ as Azumaya algebra over its centre (Proposition \ref{thm:D-Azumaya}).

\begin{thm}[Cartier descent]\label{thm:Cartier}
Let $X$ be a smooth scheme. The functor $Fr_{\nabla}^*$ induces an equivalence between the category of quasi-coherent sheaves on $X^{(1)}$ and $D$-modules on $X$ with \emph{vanishing $p$-curvature}.
\end{thm}

\begin{proof}
Let $E$ be a $D_X$-module. We remind the reader of Lemma \ref{lemma:D}, which established a correspondence between $D_X$-modules $E$ and $\D_X$-modules $F$. Moreover, $\D_X$ was shown to be an Azumaya algebra in Theorem \ref{thm:D-Azumaya}. 

Interpreting $\Psi_{\nabla}$ as action of the centre $Z(Fr_*(D_X)) = \pi^{(1)}_*\Oo_{T^*X^{(1)}}$ (Theorem \ref{thm:D-Azumaya})  on $Fr_*E$, we see that $\Psi_{\nabla} \equiv 0$ is equivalent to the $\D_X$-module $F$ being supported on the zero fibre $X^{(1)} \hookrightarrow T^*X^{(1)}$. 

The $\D_X$-module $S$ associated to the trivial local system $(\Oo_X,d)$ induces a splitting of $\D_X$ when restricted to $X^{(1)} \hookrightarrow T^*X^{(1)}$. In particular, we see that $F$ can be uniquely written as $F \cong S \otimes E'$. The projection formula reveals an equivalence $$Fr_*E \cong S \otimes E' \cong Fr_*(\Oo_X) \otimes E' \cong Fr_*(Fr^*E'),$$ respecting the $Fr_*D_X$-module structure. We can therefore conclude that $E \cong Fr^*E'$.

The construction of the canonical connection in Lemma \ref{lemma:canonical-connection} allows us to conclude that this is an equivalence of categories as claimed.
\end{proof}

The next paragraph summarises the definition of the Hitchin map of Laszlo--Pauly \cite{MR1810690}. We refer the reader to \emph{loc. cit.} for a more detailed account of this approach. In Definition \ref{defi:HdR} we will redefine the Hitchin map for local systems, using the BNR correspondence (Proposition \ref{prop:local-bnr}).

As we have seen in Corollary \ref{cor:p-linear}, the object $\Psi_{\nabla}$ is a map $E \to E \otimes \Fr^*\Omega_{X^{(1)}}^1$, in resemblance with the definition of Higgs bundles. One can show that the characteristic polynomial of the $p$-curvature $\Psi_{\nabla}$ of a connection, is itself the pullback of an element of $\A^{(1)}$, the Hitchin base of the orbicurve $X^{(1)}$. In order for this statement to make sense, we recall that $k$-points of the affine space $\A^{(1)}$ correspond to elements of the vector space $$\bigoplus_{i = 0}^{n-1}H^0(X^{(1)},\Omega_{X^{(1)}}^1).$$ This follows from \cite[5.2.3]{Katz:fk}. A direct proof is given in \cite[Prop. 3.2]{MR1810690}.
\begin{defi}\label{defi:Local_Hitchin}
The Hitchin map for local systems, as defined by Laszlo--Pauly will be denoted by $\HdR\colon \Loc(X) \to \A^{(1)}$.
\end{defi}

It has been shown in Proposition 5.1 of \emph{loc. cit.} that the zero fibre of $\HdR$ restricted to the semistable locus, that is, the stack of semistable \emph{nilpotent connections}, is universally closed.

\begin{thm}[Laszlo--Pauly]\label{thm:laszlo-pauly}
The stack of semistable nilpotent connections $\invHdR(0)^{ss}$ is universally closed.
\end{thm}

We will generalise this result to the full Hitchin map $\HdR$, using the BNR correspondence developed in the next section.

To illustrate this theory we may consider the example\footnote{The author thanks Christian Pauly for explaining this example to him.} of line bundles on an elliptic curve $X$. The moduli space of flat degree zero line bundles on $X$ is a group scheme. It arises as an extension of the Jacobian $J_X$ by the one-dimensional vector space $H^0(X,\OmegaX)$. The Hitchin morphism $$\chi_{\dR}\colon \Loc \to \A^{(1)}$$ maps down to the one-dimensional vector space $\A^{(1)} = H^0(X^{(1)},\OmegaXone)$. There is an action of $J_{X^{(1)}}$ on the moduli space of rank $1$, degree $0$ local systems $\Loc_{,1,0}(X)$, respecting the morphism $\HdR$. The pullback $Fr^*L$ of a line bundle $L$ on $X^{(1)}$ is endowed with a canonical connection $\nabla^{can}$ of zero $p$-curvature, as we have seen in Lemma \ref{lemma:canonical-connection}. In particular, we can tensor an arbitrary rank one local system $(E,\nabla)$ with $(Fr^*L,\nabla^{can})$, leaving the $p$-curvature of $(E,\nabla)$ invariant. The Azumaya-algebra viewpoint of chapter \ref{D} provides us with an alternative understanding of this action. According to Lemma \ref{lemma:D}, we can think of $\Loc_{,1,0}(X)$ as the stack of splittings of $\D_X$ relative to the family of curves $X^{(1)} \times \A^{(1)} \to \A^{(1)}$. As we have seen in Lemma \ref{lemma:split_line}, the stack of splittings is naturally acted on by $\Pic(X^{(1)} \times \A^{(1)}/\A^{(1)})$.

\begin{ex}\label{ex:local-line}
For an elliptic curve $X$, the $\A^{(1)}$-stack $\Loc_{,1,0}(X)$ of rank $1$, degree $0$ local systems is \'etale locally equivalent to the $\A^{(1)}$-stack $\Higgs_{,1,0}(X^{(1)})$ of rank $1$, degree $0$ Higgs bundles. 
\end{ex}

Before stating this example we have already illustrated how to give a proof of Example \ref{ex:local-line} using the Azumaya picture of differential operators. Below we give a more elementary discussion.
\begin{proof}[Proof of Example \ref{ex:local-line}]
The proof proceeds by constructing a section of the Hitchin map $\HdR$, after an \'etale base change. As \emph{ansatz} we study flat connections with given $p$-curvature on the trivial line bundle $(\Oo,d+\omega)$. This corresponds to finding a right-inverse to the map $H^0(X,\Omega^1) \to \Loc \to H^0(X^{(1)},\OmegaX)$. According to formula 2.1.16 in \cite{MR565469} this map is the sum of a $p$-linear and a linear map of vector spaces. Without loss of generality we may assume it is the map $\lambda \mapsto \lambda^p - \lambda$, known as the \emph{Artin-Schreier} morphism, that is, it is \'etale. By construction, we see that the codomain of the map $\A \to \A^{(1)}$ parametrises a family of connections $(\Oo_X,d+\omega)$ on the trivial vector rank one bundle with prescribed $p$-curvature. This allows us to conclude that after base-change of $\chi_{\dR}$ along this \'etale map, there exists a section $s$ of $\chi_{\dR}$. 

Using this section we can construct a morphism $\Loc_{,1,0}(X) \times_{\A^{(1)}} \A \to J_{X^{(1)}}.$ A pair consisting of a local system $(E,\nabla)$ with $p$-curvature $a' \in \A^{(1)}$, and an Artin-Schreier lift $a$ of $a'$, is sent to $(\Oo,d+\omega_a)^{\vee} \otimes (E,\nabla),$ which is itself a rank one local system with $p$-curvature zero. By Cartier descent (Theorem \ref{thm:Cartier}), it corresponds to a unique line bundle on $X^{(1)}$.
This construction gives rise to a map $\Loc_{,1,0}(X) \times_{\A^{(1)}} \A \to \Higgs_{,1,0}(X^{(1)}) \times_{\A^{(1)}} \A.$ An obvious inverse to this map can now be constructed, by using the map $\A \times \Pic(X^{(1)}),$ which sends $(a,L)$ to the local system $(\Oo,d+\omega_a) \otimes (Fr^*L,\nabla^{can})$.
\end{proof}

\subsection{The BNR correspondence}\label{local-bnr}

In its simplest form, the BNR correspondence for Higgs bundles relates the groupoid of Higgs bundles $(E,\theta)$ with characteristic polynomial $a$, to certain coherent sheaves $L$ on the spectral curve $Y_a$. In this subsection, we will prove a similar characterisation for local systems. In the case of smooth spectral curves this description was contained in the proof of Lemma 4.8 in \cite{Bezrukavnikov:fr}.

We denote by $\pi\colon Y^{(1)} \to X^{(1)} \times \A^{(1)}$ the universal spectral cover of $X^{(1)}$ parametrised by $\A^{(1)}$, where the superscript ${}^{(1)}$ denotes the Frobenius twist of a scheme, as defined in Subsection \ref{ssec:D}. To avoid cluttering our notation with indices, we do not decorate $\pi$ with a superscript, as in $\pi^{(1)}$.

\begin{prop}[BNR for local systems]\label{prop:local-bnr}
Giving an $S$-family of local systems of rank $n$ on a smooth Deligne-Mumford stack of dimension $d$ is equivalent to giving a morphism $a\colon S \rightarrow \A^{(1)}$, a coherent sheaf $\Ee$ on $Y^{(1)} \times_{\A^{(1)}} S$, carrying a structure of a $\D_{X}$-module, and satisfying that $(\pi_S)_* \Ee$ is locally free Higgs bundle on $Y^{(1)}$ of rank $p^d n$ and characteristic polynomial $a^{p^d}$.
\end{prop}

Before delving into the proof below, we will sketch the main ideas and difficulties. We use the equivalence of categories of Lemma \ref{lemma:D} $$\QCoh(X,D_X) \cong \QCoh(T^*X^{(1)},\D_X),$$ where $\D_X$ is an Azumaya algebra on $T^*X^{(1)}$. Recall that under this equivalence, a local system $(E,\nabla)$ is sent to the $Fr_*\D_X$-module $Fr_*E$, which gives rise to a $\D_X$-module $\Ee$. The $p$-curvature endows the rank $p^dn$ vector bundle $Fr_*E$ on $X^{(1)}$ with a Higgs field. The BNR correspondence for Higgs bundles (Theorem \ref{thm:bnrHiggs}) allows us now to relate the Higgs bundle $Fr_*E$ to a coherent sheaf $\mathcal{F}$ supported on a spectral curve $Y^{(1)}_b$. The degree of the polynomial $b$ equals the rank $p^dn$ of $Fr_*E$. 

As in the proof of \emph{loc. cit.} the main difficulty arises when showing that the characteristic polynomial $b$ of the Higgs bundle $Fr_*E$ together with the $p$-curvature as Higgs field, can be written as $b = a^p$. In \emph{loc. cit.} this is shown by a computation, while our proof below makes use of the fact that $D_X$ gives rise to an Azumaya algebra $\D_X$ on $T^*X^{(1)}$.

\begin{proof}[Proof of Proposition \ref{prop:local-bnr}]
The coherent sheaf $Fr_*E$ is locally free of rank $p^dn$, where $d = \dim X$. The $p$-curvature of $\nabla$ endows $Fr_*E$ with a Higgs field. An S-family $(E,\nabla)$ of local systems gives therefore rise to a morphism $$b\colon S \to \A_{p^dn}^{(1)},$$ that is, a characteristic polynomial of degree $p^dn$. According to the BNR correspondence for Higgs bundles (Theorem \ref{thm:bnrHiggs}) the sheaf $\Ee$ on $S \times T^*X^{(1)}$ is supported on the spectral cover $Y^{(1)}_b$ corresponding to the characteristic polynomial $b$.

Since $\rk Fr_*E = p^dn$ we have $\deg b = p^dn$, which is a $p$-th multiple of the degree. We need to show that $b = a^{p^d}$ for a degree $n$ characteristic polynomial $a$, and that $\Ee$ is supported on the corresponding spectral cover. This relation refers to the product of characteristic polynomials, where the coefficients are viewed as elements of the graded ring $$\bigoplus_{i=0}^{\infty}H^0(X^{(1)},(\OmegaXone)^{\otimes i}).$$

Since support can be checked \'etale locally, we may assume that $X$ is a smooth scheme. Let $x$ be a geometric point of $X \times S$ and $U = \Spec \Oo$ the spectrum of the corresponding henselian ring. We consider the base change $Y^{(1)} \times_X U = V \to U$. Because this arrow is a finite morphism and $\Oo$ is henselian, we conclude that $V$ is the spectrum of a product of local henselian algebras (\cite[Thm 4.2(b)]{MR559531}). In particular we know that $\D|_V$ is split. There exists an isomorphism $\D|_V \cong \Eend(M)$, where $M$ denotes a rank $p^d$ vector bundle on $V$. Since $\Gamma(\Oo_V)$ is a product of local rings, $M$ is free. Thus we may identify $\D|_V$ with the matrix algebra $M_{p^d}(\Oo_V) = \Eend(\Oo_V^{\oplus p^d}),$ that is, $\Oo_V^{\oplus p^d}$ is a splitting for $\D|_V$. 

This implies the existence of a coherent sheaf $\mathcal{F} = \delta_{\Oo_V^{\oplus p^d}}(\Ee)$ (where we use the notation of Lemma \ref{lemma:morita}), such that we have $$\Ee \cong \bigoplus_{i=1}^{p^d}\mathcal{F},$$ which in turn implies a decomposition of the Higgs bundle $Fr_*E$ as a direct sums of $p^d$ copies of the same Higgs bundle $(\widetilde{E},\widetilde{\theta})$ (we denote by $\widetilde{\theta}$ the Higgs field induced by the BNR correspondence). In particular we obtain that $\rk \widetilde{E} = n$ and therefore that $(\widetilde{E},\widetilde{\theta})$ is supported on the spectral cover corresponding to a degree $n$ polynomial $a$, satisfying $b = a^{p^d}$.
We conclude that $\mathcal{F}$ and $\Ee$ are supported on the spectral cover $Y^{(1)}_a$.
\end{proof}

One consequence of the above discussion is the existence of a canonical morphism from the stack of flat connections $\Loc$ to the affine space $\A^{(1)}$, providing an alternative to Definition \ref{defi:Local_Hitchin}. Recall that we have seen in Proposition \ref{prop:local-bnr} that the data of a local system can be encoded by a $\D_X$-module $\Ee$, supported on a spectral cover. The morphism defined below simply forgets the sheaf $\Ee$, respectively sends it to its scheme-theoretic support.

\begin{defi}\label{defi:HdR}
The morphism $\HdR\colon \Loc(X) \to \A^{(1)}$
$$(a\colon S \to \A^{(1)}, \Ee) \mapsto (a\colon S \to \A^{(1)})$$
 is called the \emph{twisted Hitchin morphism}.
\end{defi}

%

\subsection{Relative splittings of Azumaya algebras}\label{splitting}

Our main result in this subsection is a generalisation of a classical theorem, which is established, using a result in Galois theory due to Tsen. We refer the reader to Example 2.22 (d) in \cite{MR559531}) for a more detailed exposition.

\begin{thm}[Tsen]\label{thm:tsen}
Let $X$ be a smooth proper curve defined over an algebraically closed field $k$ of arbitrary characteristic. Then we have $H^2_{\text{\'et}}(X,\G_m) = 0$, hence every Azumaya algebra defined over $X$ splits.
\end{thm}

We have seen at the end of Subsection \ref{ssec:azumaya} that an Azumaya algebra is split if and only if the corresponding class in $H^2_{\text{\'et}}(X,\G_m)$ vanishes. Therefore the first assertion above implies the second one. In this particular instance, also the converse is true. Theorem 2.16 of \cite{MR559531} implies every class of $H^2_{\text{\'et}}(X,\G_m)$ on a quasi-compact scheme can be represented by an Azumaya algebra, when restricted to the complement of a closed subscheme of codimension $> 1$. In particular we see that on a curve, every class in $H^2_{\text{\'et}}(X,\G_m)$ arises from an Azumaya algebra.

Tsen's Theorem \ref{thm:tsen} will be generalised in this section in two different directions. First of all we remove the smoothness assumption, and secondly, we further relativise this statement. Moreover, we allow our curves to acquire tame orbifold points. The validity of Tsen's theorem in an orbifold context is guaranteed by a result of F. Poma (\cite[Cor. 4.15]{MR3016518}).

\begin{defi}\label{defi:relative-splitting}
Let $X$ be a scheme, given a morphism of stacks $\pi\colon Y \rightarrow X$ and an Azumaya algebra $\D$ over $Y$ we say that $\D$ \emph{splits relatively over $X$}, if there exists an \'etale covering $(U_i)_{i \in I}$ of $X$, such that $\D$ is split over every fibre product $U_i \times_X Y$.
\end{defi}

Our main result in this subsection is a generalisation of Theorem \ref{thm:tsen}. We will remove the smoothness assumption and further prove the corresponding statement relative to a more general base.

\begin{defi}\label{defi:relative_curve}
A \emph{relative orbicurve} is a morphism of finite type tame Deligne-Mumford stacks $\pi\colon Y \to X$ over $k$, which is proper, flat, cohomologically flat in degree zero, and whose geometric fibres are orbicurves (Definition \ref{defi:orbicurve_new}).
\end{defi}

We can now state the main result of this section.

\begin{thm}\label{thm:relative_curve}
Let $\pi\colon Y \rightarrow X$ be a relative orbicurve, and $\D$ an Azumaya algebra on $Y$. Then every Azumaya algebra over $Y$ splits relatively over $X$.
\end{thm}

We begin by treating the absolute case (that is, with base being $\Spec k$), which is stated in \cite{MR926276}, in a remark after Lemma I.5.2.

\begin{thm}\label{thm:curve}
Let $X$ be a proper noetherian $DM$-stack over an algebraically closed field $k$ of dimension $\leq 1$, then $H^2_{\text{\'et}}(X,\G_m) = 0$. In particular, every Azumaya algebra defined over $X$ splits.
\end{thm}

\begin{proof}
At first we reduce to the case where $X$ is reduced. Let $\mathcal{I}$ be a quasi-coherent sheaf of ideals on $X$ satisfying $\mathcal{I}^2=0$. We consider the closed immersion $j\colon Y \rightarrow X$ given by this sheaf of ideals in $\Oo_X$ and study the truncated exponential sequence
\[
\xymatrix{ 0 \ar[r] & \mathcal{I} \ar[r]^{\exp} &  \Oo_X^{\times} \ar[r] & j_*\Oo_{Y}^{\times} \ar[r] & 1. }
 \]
The exponential function $\exp\colon \mathcal{I} \rightarrow  \Oo_X^{\times}$ is defined by the expression $f \mapsto 1 + f$ and satisfies $\exp(f+g) = \exp(f)\exp(g)$, $\exp(0) = 1$. In particular the map takes values in the sheaf of abelian groups of units. The corresponding long exact sequence implies that $H^2_{\text{\'et}}(X, \G_m) = H^2_{\text{\'et}}(Y,\G_m)$, since $H^i(X, \mathcal{I}) = 0$ for $i > 1$. There exists a quasi-coherent sheaf of ideals $\mathcal{J}$ on $X$, such that $\Oo_X/\mathcal{J} = \Oo_{X^{red}}$. Since $\mathcal{J}$ consists of nilpotent elements and $X$ is noetherian, there exists a positive integer $k$, such that $\mathcal{J}^k = 0$. Let us define $X_i\hookrightarrow X$ to be the closed immersion given by the sheaf of ideals $\mathcal{J}^{i}$. By definition we have $X_k = X$ and $X_0 = X^{red}$. Moreover $\mathcal{I}_i=\mathcal{J}^{i-1}/\mathcal{J}^i$ defines a quasi-coherent sheaf of ideals on $X_i$ satisfying $\mathcal{I}_i^2 = 0$. From the discussion above we may conclude that $$H^2_{\text{\'et}}(X_k, \G_m) = \cdots = H^2_{\text{\'et}}(X_0,\G_m).$$

Assuming that $X$ is reduced, we denote by $j\colon Y \rightarrow X$ the normalization map. The morphism $j$ is finite, therefore the functor $j_*$ is exact (cf. \cite[Cor II.3.6]{MR559531}). We can now apply the \'etale cohomology functor to the short exact sequence
$$ 1 \rightarrow \Oo_X^{\times} \rightarrow  j_*\Oo_Y^{\times} \rightarrow j_*\Oo_Y^{\times}/\Oo_X^{\times} \rightarrow 1 $$
and obtain a long exact sequence
$$ \dots \rightarrow H^i_{\text{\'et}}(X,\Oo_X^{\times}) \rightarrow H^i_{\text{\'et}}(Y,\Oo_Y^{\times}) \rightarrow H^i(j_*\Oo_Y^{\times}/\Oo_X^{\times}) \rightarrow H^{i+1}_{\text{\'et}}(X,\Oo_X^{\times}) \rightarrow \dots \;.$$

The quotient sheaf $ j_*\Oo_Y^{\times}/\Oo_X^{\times}$ is supported at finitely many closed points, therefore all of its higher cohomology groups have to vanish. As a consequence we obtain that $H^2_{\text{\'et}}(X,\Oo_X^{\times}) = 0$.

As has been argued above, we may assume that $X$ is smooth, and therefore we are in the situation of Theorem \ref{thm:tsen}.
\end{proof}

Recall that a splitting of an Azumaya algebra $\D$ on $X$ is a pair $(\phi,E)$, where $E$ is a locally free sheaf on $X$ and $$ \phi\colon \Eend(E) \to \D $$ is an isomorphism.

\begin{defi}
Given an Azumaya algebra $\D$ over $Y$ and a morphism $Y \to X$, we define a $2$-functor $\Ssplit$ from the category $\Sch/X$ to the $2$-category of groupoids sending a scheme $U \rightarrow X$ to the groupoid of splittings of $\D$ over $U \times_X Y$. Here a morphism between two splittings $(\phi,E)$, $(\psi,F)$ is defined to be a pair $(\gamma,L)$ where $L$ is a line bundle on $U \times_X Y$ and $\gamma\colon E \rightarrow F \otimes L$ is an isomorphism. This is a stack referred to as the stack of relative splittings of $\D$ along $\pi\colon Y \to X$.
\end{defi}

To deduce the relative from the absolute case, we need to study the deformation theory of splittings.

\begin{lemma}[Commutation with inverse limits]\label{lemma:inverse_limits}
Let $\pi\colon Y \rightarrow X$ be a relative orbicurve (Definition \ref{defi:relative_curve}) and $\D$ an Azumaya algebra over $Y$. Given a complete noetherian local $k$-algebra $\bar{A}$ together with a morphism $\Spec \bar{A} \rightarrow X$, then $\D$ splits over the base change $\Spec \bar{A} \times_X Y$. Furthermore, we obtain for the $2$-functor $\Ssplit$ that
$$ \Ssplit(\bar{A}) \rightarrow \invlim \Ssplit(\bar{A}/\mathfrak{m}^n) $$
is an isomorphism.
\end{lemma}
\begin{proof}
In the course of the proof we will assume that $X = \Spec \bar{A}$. We denote the base change $Y \times_X \Spec A/\mathfrak{m}_A^n$ by $Y_n$. The corresponding formal scheme is denoted by $\widehat{Y}$.

Grothendieck's Existence Theorem states that the abelian category of coherent sheaves on $Y$ is equivalent to the abelian category of coherent sheaves on the formal scheme $\widehat{Y}$, which is a projective 2-limit category of the categories $\Coh(Y_n)$ (cf. \cite[Thm 8.4.2]{MR2223409} for the case of schemes, and \cite[App. A]{MR1862797} for tame Deligne-Mumford stacks). According to Theorem \ref{thm:curve} there exists a splitting $S_n$ of $\D|_{Y_n}$ for every $n$. In order to obtain a splitting of $\D$ we need to choose compatible splittings $S_n$. This is always possible, as the splitting $S_{n+1}|_{Y_n}$ differs from $S_n$ by a twist by a line bundle (see Lemma \ref{lemma:split_line}). Since $\pi\colon Y \to X$ is a relative orbicurve, we can lift the difference line bundle from $Y_n$ to $Y_{n+1}$. Therefore, a sequence of compatible splittings $S_n$ of $\D_{Y_n}$ exists.

To prove the asserted equivalence, we use the existence of a splitting and reduce to the analogous statement for line bundles. This is an easy consequence of Grothendieck's Existence Theorem quoted above, since line bundles are characterised as invertible coherent sheaves.
\end{proof}

We refer the reader to \cite{MR0399094} for the definition of the deformation theory of a stack. Note that in our case the deformation problem is unobstructed, as we are dealing with a relative orbicurve. According to a Theorem of M. Artin (cf. \cite[Thm 5.3]{MR0399094}), the deformation theory of a stack allows us to decide wether it is algebraic or not.

\begin{thm}
Let $\pi\colon Y \rightarrow X$ be a relative orbicurve and let $\D$ be an Azumaya algebra over $Y$. The $2$-functor $\Ssplit$ is representable by an algebraic stack.
\end{thm}
\begin{proof}
The stack $\Split$ is by definition equivalent to the \emph{Hom-stack} $\Map_{\A^{(1)}}(Y^{(1)},\Y_{\D_{\Bun}})$. Algebraicity of such Hom-stacks has been shown under quite general assumptions by M. Aoki in \cite{MR2194377}, respectively \cite{MR2258535}. The family of spectral curves $Y^{(1)}/\A^{(1)}$ is known to be proper, and Lemma \ref{lemma:inverse_limits} guarantees that the additional assumption of \cite{MR2258535} is satisifed.
\end{proof}

Spelling out an alternative proof of algebraicity, one uses the fact that $\Split$ is a $\Pic$ quasi-torsor. By this we mean that if a splitting exists, then it gives rise to an identification of the groupoid of splittings with the groupoid of line bundles. Moreover Lemma \ref{lemma:inverse_limits} shows that $\D$ splits on formal fibres, rendering the restriction of $\Split$ to formal fibres a $\Pic$-torsor. It has been shown in \cite{MR0260746} that $\Pic$ is an algebraic stack, by a close study of the deformation theory of $\Pic$. But since $\D$ splits on formal fibres, $\Pic$ and $\Split$ cannot be distinguished on the level of deformation theory, which allows us to conclude that $\Split$ is an algebraic stack too. 

\begin{lemma}\label{lemma:smooth}
Let $\pi\colon Y \rightarrow X$ be a relative orbicurve and let $\D$ be an Azumaya algebra over $Y$. Then the algebraic stack $\Split$ is smooth over $X$.
\end{lemma}
\begin{proof}
This is a simple verification of the criterion for formal smoothness. We have already seen that there exists no obstruction to lifting splittings in the curve case in the proof of Lemma \ref{lemma:inverse_limits}. Therefore the structural morphism $\Ssplit \rightarrow X$ is smooth.
\end{proof}

\begin{proof}[Proof of Theorem \ref{thm:relative_curve}]
According to Lemma \ref{lemma:smooth}, the morphism of stacks $\Split \to X$ is smooth. In particular, it has a section \'etale locally on $X$.
\end{proof}

\begin{lemma}\label{lemma:fpuo}
The stack $\Ssplit$ is \'etale locally isomorphic to $\Pic(Y/X)$. Moreover, it is locally of finite presentation, smooth and universally open over the base.
\end{lemma}
\begin{proof}
We have seen in Theorem $\ref{thm:relative_curve}$ that every Azumaya algebra splits relatively over $X$. Therefore there exists an \'etale cover $(U_i)$ of $X$, such that $\D$ splits over $U_i \times_X Y$. Since two splittings of an Azumaya algebra are related by line bundles this gives rise to an isomorphism $U_i \times_X \Ssplit \cong \Pic(U_i \times_X Y/U_i)$. 

See Proposition 9.4.17 in \cite{MR2223410} for a proof that $\Pic(Y/X)$ is locally of finite presentation. In \emph{loc. cit.} it is assumed that $X$ is a curve but their proof generalises verbatim to the case of orbicurves. Moreover $\Pic(Y/X)$ is smooth as the deformation theory of line bundles on curves is unobstructed. A flat morphism which is locally of finite presentation is universally open, according to Proposition 2.4.6 in EGA IV.2 (cf. \cite{MR0199181}).
\end{proof}

\subsection{Local equivalence of moduli stacks}\label{equivalence}

Let $X$ be a smooth projective orbicurve. In this subsection we use the BNR correspondence \ref{prop:local-bnr} and Theorem \ref{thm:relative_curve} to show that the moduli stack of Higgs bundles is \'etale locally equivalent to the moduli stack of local systems over the Hitchin base $\A^{(1)}$. This is based on the following observation.

\begin{lemma}[Splitting principle]\label{lemma:splittingprinciple}
Let $S$ be a $k$-scheme locally of finite type and $a\colon S \to \A^{(1)}$ be an $S$-family of spectral curves, such that the Azumaya algebra $\D_X$ has a splitting $P$, when pulled back to $Y^{(1)}_a = Y^{(1)} \times_{\A^{(1)}} S$. Then the equivalence $\delta_P$ from Lemma \ref{lemma:morita} induces an equivalence of the groupoid of rank $n$ local systems on $X$ with characteristic polynomial $a$ and rank $n$ Higgs bundles on $X^{(1)}$ with characteristic polynomial $a$.
\end{lemma}

\begin{proof}
We denote by $\pi\colon Y^{(1)}_a \to X^{(1)} \times S$ the projection of the corresponding spectral curve to $X \times S$, which is a finite morphism. We will use the BNR correspondence for local systems (Prop. \ref{prop:local-bnr}) and the BNR correspondence for Higgs bundles (Thm. \ref{thm:bnrHiggs}). In particular, we would like to relate $\D$-modules $\Ee$ on $Y^{(1)}_a$, satisfying $\pi_*\Ee$ being a locally free sheaf of rank $pn$, with coherent sheaves $\F$, such that $\pi_*\F$ is a locally free sheaf of rank $n$. As in Lemma \ref{lemma:morita} we send such a sheaf $\F$ to $$\Ee= P \otimes \F.$$Since Lemma \ref{lemma:morita} already provides us with an equivalence of quasi-coherent $\D_X$-modules with quasi-coherent sheaves on $Y_a$, we only have to take care of the push-forward condition to conclude the proof. This is an \'etale local property, and hence we may assume that $X$ is a scheme. For every $x \in X \times S$ we replace $X \times S$ by the spectrum of the henselization of the local ring $\Oo_{X\times S,x}$. In particular, we have that the base change of the spectral cover $Y^{(1)}_a$ is the spectrum of a product of local rings (Thm. 4.2 in \cite{MR559531}). This implies that $P$ is a free sheaf of rank $p$. In particular we see that on the level of underlying sheaves we have $$\Ee \cong \bigoplus_{i=0}^{p-1} \F.$$ We may therefore conclude that $\pi_*\Ee$ is locally free of rank $pn$ and characteristic polynomial $a^p$, if and only if $\pi_*\F$ is locally free of rank $n$ and characteristic polynomial $a$.
\end{proof}

As an immediate corollary of this splitting principle we obtain the existence of an isomorphism of Hitchin fibres for local systems and Higgs bundles.

\begin{cor}\label{cor:fibres}
Let $\HDol\colon \Higgs(X^{(1)}) \rightarrow \A^{(1)}$ be the Hitchin fibration mapping a Higgs bundle to the characteristic polynomial of its Higgs field, and $\HdR\colon \Loc \rightarrow \A^{(1)}$ the Hitchin fibration of Definition \ref{defi:HdR}. Then we have that for every $a \in \A^{(1)}$ there exists a non-canonical isomorphism of stacks $\invHDol(a) \cong \invHdR(a)$.
\end{cor} 

\begin{proof}
This follows from applying the Splitting Principle \ref{lemma:splittingprinciple} to the morphism $a\colon \Spec k \to \A^{(1)}$. The existence of a splitting of an Azumaya algebra on a non-necessarily smooth but projective orbicurve, is guaranteed by Theorem \ref{thm:curve}.
\end{proof}

Let $X$ be a complete smooth orbicurve over an algebraically closed field $k$ of characteristic $p > 0$. We denote by $\HDol\colon \Higgs(X^{(1)}) \rightarrow \A^{(1)}$ the Hitchin fibration mapping a Higgs bundle to the characteristic polynomial of its Higgs field, and $\HdR\colon \Loc \rightarrow \A^{(1)}$ the twisted Hitchin fibration of Definition \ref{defi:HdR}. These two morphisms induce the structure of an $\A^{(1)}$-stack on their domains. 

\begin{thm}\label{thm:main}
Let us denote by $\Split$ the stack of splittings of the Azumaya algebra $\D_X$ relative to the family of spectral curves $Y^{(1)}/\A^{(1)}$ (see the sections \ref{local-bnr} and \ref{splitting}). There exists a canonical isomorphism
$$\Delta\colon \Split \times_{\A^{(1)}} \Loc \cong \Split \times_{\A^{(1)}} \Higgs. $$
Moreover the $\A^{(1)}$-stacks $\Loc(X)$ and $\Higgs(X^{(1)})$ are \'etale locally equivalent, that is, there exists an \'etale cover $\{U_i \to \A^{(1)}\}_{i\in I}$ and isomorphisms of $U_i$-stacks
$$\delta_S\colon U_i \times_{\A^{(1)}} \Higgs(X^{(1)}) \cong U_i \times_{\A^{(1)}} \Loc(X).$$
\end{thm}
\begin{proof}
The first part of this theorem follows from the Splitting Principle (Lemma \ref{lemma:splittingprinciple}). We remind the reader that it states that every $S$-family of spectral cover $a\colon S \to \A^{(1)}$, which is endowed with a splitting $P$ of $\D_X$ pulled back to $S \times_{\A^{(1)}} Y^{(1)}$, gives rise to an identification of $S$-families of Higgs bundles on $X^{(1)}$ with spectral cover $Y^{(1)}_a$, and $S$-families of local systems on $X$ with spectral curve $X$.

Theorem \ref{thm:relative_curve} implies that we can choose an \'etale cover $(U_i)_{i\in I}$ and splittings $S_i$ of $\D_{X}$ on $U_i \times_{\A^{(1)}} Y^{(1)}$. We obtain isomorphisms $$U_i \times_{\A^{(1)}} \Higgs(X^{(1)}) \cong U_i \times_{\A^{(1)}} \Loc(X),$$
and thus conclude the proof.\end{proof}

\subsection{Stability}\label{localstability}

We investigate the interaction of the local equivalence of moduli stacks in Theorem \ref{thm:main} with stability. Unmindful choice of a splitting in the proof of this theorem will lead to the degree of the underlying bundles to be scaled and shifted. If the spectral cover has several components these shifts might differ between the components, which could certainly mean that stability is not preserved. Nonetheless, it is possible to single out a connected component $\Split^0 \subset \Split$ of good splittings, where the degree of the underlying bundles and their Higgs subbundles will only be scaled. Our discussion will only use basic properties for the degree function of coherent sheaves. We will take a little detour first, establishing these results in the context of possibly singular, proper DM-stacks of dimension $1$.

\subsubsection*{Degree for singular orbicurves}

This brief paragraph is concerned with establishing a Riemann-Roch theorem for possibly singular orbicurves. Everything stated here could be easily deduced from \cite{MR1710187}, but the $1$-dimensional theory allows for a very elementary account, which we recall here.

\begin{defi}\label{defi:rank}
Let $Z$ be an orbicurve. For a coherent sheaf $\F$ on $Z$ of constant rank, which \'etale locally has a finite resolution $(V^{\bullet})$ by finite rank locally free sheaves, we define its \emph{rank} to be the integer-valued function $\rnk \F(x) = \sum_{i=0}^\infty \rnk V^i|_x$, where $x$ is a geometric point of $Z$.
\end{defi}

The condition that $\F$ admits \'etale locally a finite resolution by finite rank locally free sheaves is tantamount to $\F$ being quasi-isomorphic to a perfect complex in $D^b(Z)$. Since the Grothendieck group of perfect complexes $K_0(\Perf(Z))$ is isomorphic to the Grothendieck group of vector bundles $K_0(Z)$, we see that $\rnk \F$ is well-defined and locally constant.

The following definition produces the classical notion of degree for smooth complete curves, as is implied by the Riemann-Roch Theorem. Hence, this definition extends the degree function to possibly singular orbicurves.

\begin{defi}\label{defi:deg}
Let $Z$ be an orbicurve. For a coherent sheaf $\F$ on $Z$, which \'etale locally has a finite resolution by finite rank locally free sheaves, we define its \emph{degree} to be the integer $\deg \F = \deg_Z \F = \chi(Z,\F) - m\chi(Z,\Oo_Z) = \sum_{i = 0}^{1} (-1)^i (\dim H^i(Z,\F) - \dim H^i(Z,\Oo^{m}))$.
\end{defi}

The alternating sum above is finite, by virtue of cohomology vanishing in degree higher than $\dim X$. We will now investigate the behaviour of $\deg$ with respect to tensor products. We will denote by $K'_0(Z)$ the Grothendieck group of coherent sheaves on $Z$. If $Z' \subset Z$ is a closed subset, we denote by $K'_0(Z,Z')$ the Grothendieck group of coherent sheaves on $Z$, set-theoretically supported on $Z'$.

\begin{lemma}\label{lemma:deg_formula}
Suppose that $Z$ and $\F$ satisfy the assumptions of Definition \ref{defi:deg}, and let $V$ be a locally free sheaf of rank $\ell = \rk V$ on $Z$. Then, we have $\ell\deg(\F) = \chi(Z,\F\otimes V) - \chi(V^{\oplus m})$. In particular, if $V$ is a line bundle, we see that $\deg(\F) = \chi(Z,\F\otimes V) - \chi(V^{\oplus m})$.
\end{lemma}
\begin{proof}
The last assertion is a special case of the first. We choose an open subset $U \subset Z$ contained in the schematic part of $Z$, and an isomorphism $\F|_U \cong \Oo_Z^{m}$. Moreover, we may assume that $X \setminus U$ is of dimension $0$. Since we have an exact sequence
$$K'_0(X,X\setminus U) \rightarrow K'_0(X) \rightarrow K'_0(U),$$
and the element $[\F] - m[\Oo_Z]$ is sent to $0$ in $K'_0(U)$, we see that it lies in the image of $K'_0(X,X \setminus U)$. In particular, the set-theoretic support of this $K$-theory class is $0$-dimensional. This implies that $\chi(Z,\F\otimes V) - m\chi(V) = [V]\otimes ([\F] - m[\Oo_Z]) = [\Oo_Z^{\ell}] \otimes ([\F] - m[\Oo_Z]) = \ell \deg \F$.
\end{proof}

\begin{cor}\label{cor:deg_formula}
If $Z$ satisfies the conditions of Definition \ref{defi:deg}, and $L$, $M \in \Pic(Z)$, we have $\deg(L \otimes M) = \deg L + \deg M$. If $V$ is a locally free sheaf of rank $\ell$, then $\deg(L \otimes V) = \ell \deg L + \deg V$.
\end{cor}
\begin{proof}
The first formula is a special case of the second. We compute $\deg(L \otimes V)$ as 
$$\chi(L \otimes V) - \chi(\Oo_Z^{\ell}) = (\chi(L \otimes V) - \ell\chi(L)) + (\ell\chi(L) - \ell\chi(\Oo_Z)) = \deg(V) + \ell \deg L.$$
Where we identified the left hand term in brackets with $\deg V$ using Lemma \ref{lemma:deg_formula}, and the right hand term with $\ell \deg L$, by virtue of Definition \ref{defi:deg}.
\end{proof}

\begin{rmk}\label{rmk:deg}
Let $\Z \rightarrow S$ be a flat morphism of Deligne-Mumford stacks, fibrewise satisfying the assumption of Definition \ref{defi:deg}, and choose a coherent $S$-flat sheaf $\F$ on $\Z$, which satisfies the assumption of Definition \ref{defi:deg} fibrewise. We assume that $\Z \rightarrow S$ can be factored as $\Z\rightarrow \Z' \rightarrow Z' \rightarrow S$, where $f\colon \Z\rightarrow \Z'$ is a finite and flat map, which guarantees $\chi(\Z_s,\F) = \chi(\Z'_s,f_*\F)$, $\Z \rightarrow Z'$ is the map to the coarse moduli space, and the map $Z \rightarrow S$ is projective and flat. It then follows from \cite[Lemma 4.2]{olsson2003quot} that the function $s \mapsto \chi(\Z_s,\F)$ is locally constant. Hence we conclude:
\begin{itemize}
\item[(a)] The function $s \mapsto \deg_{\Z_s}\F$ is locally constant on $S$.
\item[(b)] We define the arithmetic genus $g(Z)$ to be $1-\chi(Z,\Oo_Z)$. As above, the function $s \mapsto g(\Z_s)$ is locally constant on $S$.
\end{itemize}
\end{rmk}

\subsubsection*{The case of spectral covers}

In order to apply degree calculus introduced in the preceding paragraph to spectral covers, we need to ensure that the coherent sheaves which arise via the BNR correspondence admit a finite resolution by finite rank locally free sheaves.

\begin{lemma}
Let $\pi\colon Y_a \rightarrow X$ be a spectral cover for a smooth orbicurve $X$, and $\F$ a coherent sheaf on $Y_a$, such that $\pi_*\F$ is a locally free sheaf on $X$. Then, $\F$ admits a finite resolution by finite rank locally free sheaves. 
\end{lemma}

\begin{proof}
Since $X$ is smooth, every finite rank locally free sheaf is a maximal Cohen-Macaulay sheaf. The morphism $\pi\colon Y_a \rightarrow X$ is finite, hence $\F$ is a maximal Cohen-Macaulay sheaf as well. This implies that as an object in $D^b(Y_a)$ we have $$\mathbb{D}\F = R\Hhom(\F,\omega_{Y_a}) \simeq \F.$$ The spectral cover $Y_a$ is a locally complete intersection, since it is of dimension $1$, and is defined as the vanishing scheme of the section 
$\lambda^n + a_{n-1} \lambda^{n-1} + \cdots + a_0 \in H^0(T^*X,(\Omega_X)^{\otimes n})$. Therefore we have that the canonical sheaf $\omega_{Y_a}$ is invertible, and we deduce that $\F$ is a dualisable object of $D^b(Y_a)$. This implies that it is a perfect complex, that is, \'etale locally has a finite resolution by finite rank locally free sheaves.
\end{proof}

\begin{cor}
Let $X$ be a smooth orbicuve, and let $\Ee$ be a coherent sheaf on a spectral cover $Y_a$, corresponding to a Higgs bundle $(E,\theta)$ on $X$. Then, $\Ee$ admits \'etale locally a finite resolution by finite rank locally free sheaves. Similarly, if $\F$ denotes the $\D$-module on $Y_a^{(1)}$ corresponding to a flat connection $(F,\nabla)$ on $X$, then $\F$ admits \'etale locally a finite resolution by finite rank locally free sheaves.
\end{cor}

We also need to ensure that Remark \ref{rmk:deg}, that is, the fact that $s \mapsto \chi(\Z_s,\F|_{\Z_s})$ is a locally constant function under suitable conditions, can be applied to the family of spectral covers $Y \rightarrow \A$. We denote by $X'$ the coarse moduli space of $X$, and can factor $Y \rightarrow S$ as
$Y \rightarrow X \times S \rightarrow X' \rightarrow S,$ 
where $X' \times S \rightarrow S$ is projective and flat, and $Y \rightarrow X \times S$ is finite and flat.

\subsubsection*{Computing the degree change}

Using the previously obtained identities, we are ready to compute how the choice of a splitting affects the respective degrees of Higgs bundles and local systems.

\begin{lemma}\label{lemma:deg}
Let $a \in \A^{(1)}$ be the characteristic polynomial of a spectral curve $Y^{(1)}_a \to X^{(1)}$. Given a splitting $S$ of $\D_X$ on $Y^{(1)}_a$ the induced isomorphism of Hitchin fibres $\invHDol(a) \to\invHdR(a)$ sends a degree $d$ Higgs bundle on $X^{(1)}$ to a degree $$pd - (1-p)(1-h)n + \deg_{Y_a^{(1)}} S$$ local system on $X$, where $h$ denotes the genus of $X$.
\end{lemma}

\begin{proof}
Given a local system $(E,\nabla)$ the first step is to push it forward along the Frobenius morphism $\Fr\colon X \to X^{(1)}$. This being a finite morphism we obtain $\chi(X,E) = \chi(X^{(1)},\Fr_*E)$, and using the definition of degree (\ref{defi:deg}) and genus (\ref{rmk:deg}(b)), deduce the equality
\begin{equation}\label{eqn:Fr}
\deg_X E + n(1 - h) = \deg_{X^{(1)}} \Fr_*E + pn(1-h). 
\end{equation}
In particular we have $\deg_{X^{(1)}} \Fr_* E = \deg_X E + n(1-p)(1-h)$. If $\pi\colon Y^{(1)}_a \to X^{(1)}$ denotes the finite morphism from the spectral cover to $X$, we can write $\Fr_*E = \pi_*\Ee = \pi_*(L \otimes S)$, where $L$ is a coherent sheaf on $Y^{(1)}_a$, and we have $\delta_S(\Ee) = L$. Recall that this means that we can write $\Ee$ as $L \otimes S$, since $S$ is a splitting of $\D_X$ on $Y^{(1)}_a$. Using again the identity $\chi(X^{(1)},\pi_*(L \otimes S)) = \chi(Y^{(1)}_a,L \otimes S)$ and Corollary \ref{cor:deg_formula}, we obtain
\begin{equation}\label{eqn:SL} \deg_{X^{(1)}} \Fr_* E + pn(1-h) = p\deg_{Y^{(1)}_a} L + \deg_{Y^{(1)}_a} S + p(1-g). \end{equation}
Here we use $g$ to denote the arithmetic genus of the spectral cover $Y^{(1)}_a$. On the other hand, we compute for the Higgs bundle corresponding to $(E,\nabla)$ and $S$, that is, for $\pi_*L$ the Euler characteristic
\begin{equation}\label{eqn:L} \deg_{X^{(1)}} \pi_*L + n(1-h) = \chi(X^{(1)},\pi_*L) = \chi(Y^{(1)}_a,L) = \deg_{Y^{(1)}_a} L + (1 - g).\end{equation}
From equations (\ref{eqn:SL}) and (\ref{eqn:L}) we conclude that $p\deg_{X^{(1)}} \pi_*L + \deg_{X^{(1)}} S = \deg_{X^{(1)}} \pi_*(L \otimes S) = \deg_{X^{(1)}} Fr_*E$. Equation (\ref{eqn:Fr}) implies now that $\deg_X E = p \deg_{X^{(1)}} \pi_*L - (1-p)(1-h)n + \deg_{Y^{(1)}_a} S$.
\end{proof}

\begin{cor}\label{cor:deg_shift}
Fix $a \in \A^{(1)}$, and $S \in \Split(Y^{(1)}_a)$ a splitting of $\D_X$ on the spectral curve $Y^{(1)}_a$. The induced map $\delta_S^{-1}\colon \invHDol(a) \to\invHdR(a)$ fits into a commutative diagram
\[
\xymatrix{
\invHDol(a) \ar[r]^{\delta_S^{-1}} \ar[d]_{\deg_{X^{(1)}}} & \invHdR(a) \ar[d]^{\deg_{X}}\\
\mathbb{Z} \ar[r]^{x \mapsto px + c} & \mathbb{Z},
}
\]
with the horizontal arrow being an affine function of slope $p$ with constant term $c$ depending on $\deg_{Y^{(1)}_a} S$, $n$, $h = g(X)$, and $p$.
\end{cor}

\subsubsection*{Selecting good splittings}

\begin{lemma}\label{lemma:Pic_surjective}
A connected component $\C$ of $\Pic(Y^{(1)}/\A^{(1)})$ maps surjectively to $\A^{(1)}$ if and only if the fibre over $0$, $\C_{0}\neq \emptyset$ is non-empty. 
\end{lemma}
\begin{proof}
The morphism $\Pic(Y^{(1)}/\A^{(1)})$ is flat and locally of finite presentation. Therefore it is a universally open map. This shows that the image of every connected component of $\Pic(Y^{(1)}/\A^{(1)})$ is an open subset of $\A^{(1)}$. Moreover, $\A^{(1)}$ and $Y^{(1)}$ are endowed with a $\G_m$-action (with positive weights), hence $\Pic(Y^{(1)}/\A^{(1)})$ inherits an action as well; and by continuity also the connected component $\C$ and its image inside $\A^{(1)}$. Since $\G_m$ acts with positive weights on $\A^{(1)}$, the only $\G_m$-equivariant open subset containing $0$ is $\A^{(1)}$. 
\end{proof}

\begin{cor}\label{cor:Pic_d_surjective}
The union of connected components $\Pic^{nd}$ of $\Pic(Y^{(1)}/\A^{(1)})$, consisting of fibrewise degree $nd$ line bundles, maps surjectively to $\A^{(1)}$.
\end{cor}

\begin{proof}
According to Lemma \ref{lemma:Pic_surjective} it suffices to show that the spectral curve $Y^{(1)}_0$ supports a line bundle $L$ of degree $nd$, for every integer $d \in \mathbb{Z}$. Let $M$ be a degree $d$ line bundle on $X^{(1)}$. We recall the notation $\pi\colon Y_0^{(1)} \rightarrow X$, and will show that $\pi^{*}M$ restricts to a line bundle of degree $nd$ on $Y^{(1)}_0$.
Indeed, we find
$\deg L = \chi(\pi_*\pi^*M) - \chi(\pi_* \Oo) = \chi(M \otimes [\pi_*\Oo]) = \chi(M\otimes [\Oo_{X^{(1)}} \oplus \Omega^1_{X^{(1)}} \oplus \cdots \oplus (\Omega^1_{X^{(1)}})^{\otimes n}]^{\vee}) - \chi([\Oo_{X^{(1)}} \oplus \Omega^1_{X^{(1)}} \oplus \cdots \oplus (\Omega^1_{X^{(1)}})^{\otimes n}]^{\vee}) = n\deg M = nd.$
\end{proof}

In fact, it is not possible to do better. If $L$ is a line bundle on $Y^{(1)}_0$, then its degree is divisible by $n$. This is a consequence of the following two lemmas.

\begin{lemma}\label{lemma:first}
Let $Z$ be an orbicurve, and $L$ a line bundle on $Z$. We denote by $Z_n$ the $n$-th order infinitesimal neighbourhood of the zero section $i\colon Z\hookrightarrow \mathrm{Tot}\; L^{\vee} = \Spec_Z \Sym L$. We have $Z_0 = Z$. For a vector bundle $M$ on $Z_n$, we denote by $i_{n-1}\colon Z_{n-1}\hookrightarrow Z_n$ the canonical inclusion. We have a short exact sequence
$$0 \rightarrow L^{\otimes (n-1)} \otimes i_*i^*M \rightarrow M \rightarrow (i_{n-1})_*i_{n-1}^*M \rightarrow 0.$$
\end{lemma}

\begin{proof}
By definition, $Z_n = \Spec_Z \bigoplus_{i = 0}^{n-1}L^{\otimes i}$, where we the product structure of the algebra $\bigoplus_{i = 0}^{n-1}L^{\otimes i}$ satisfies the relation $\alpha_1 \cdot \alpha_2 = 0$ for $\alpha_j$ local sections of $L^{\otimes i_j})$ with $i_1 + i_2 \geq n$. 

Let us denote the kernel of the natural map $M \rightarrow (i_{n-1})_*i_{n-1}^*M$ by $K$, it remains to show that $K \cong L^{\otimes (n-1)}\otimes M$.

The $\Oo_{Z_n}$-module structure on $M$ induces therefore a map
$$L^{\otimes (n-1)} \otimes M \rightarrow M,$$
whose composition with $M \rightarrow (i_{n-1})_*i_{n-1}^*M$ vanishes, since $L^{\otimes(n-1)}$ acts as $0$ on $\Oo_{Z_{n-1}}$. Therefore we obtain a map $f\colon L^{\otimes (n-1)} \otimes M \rightarrow K$. Moreover, since $f(\alpha x) = 0$ for every local section $\alpha$ of $L^{\otimes i}$ with $i \geq 1$, we see that $f$ factors through a map
\begin{equation}\label{eqn:first}i_*i^*(L^{\otimes (n-1)} \otimes M) \rightarrow K.\end{equation}
It remains to show that this is an isomorphism of sheaves. This is a local problem, and we may therefore choose an \'etale covering, with respect to which $M$ and $L$ can be trivialised. Henceforth we may assume that $Z = \Spec A$, $L = A$, and $M$ a free $A$-module. In particular we have $Z_n = \Spec A[t]/(t^n)$.

The map \eqref{eqn:first} is now the canonical morphism $$t^{n-1}M \rightarrow \ker(M \rightarrow M \otimes_{A[t]/(t^n)} A[t]/(t^{n-1})).$$ This is an isomorphism if $M$ is free.
\end{proof}

\begin{lemma}\label{lemma:second}
Let $Z$, $L$, and $M$ be as in Lemma \ref{lemma:first}, where we denote the rank of $M$ by $k$. Then, we have $\deg_{Z_n}M = n\deg_Z i^*M$, and with respect to the projection $\pi\colon Z_n \rightarrow Z$ we have $\deg_Z \pi_*M = n\deg_Zi^*M + k\frac{n(n-1)}{2}\deg_ZL.$
\end{lemma}

\begin{proof}
We will deduce these two statements from Lemma \ref{lemma:first}. With the help of the aforementioned lemma, and induction, we deduce that 
$\chi(Z_n,M) = \sum_{i = 0}^{n-1}\chi(Z,i^*M\otimes L^{\otimes i}).$ In particular we see that
$$\deg_{Z_n}M = \chi(Z_n,M) - \chi(Z_n,\Oo_{Z_n}^k)=\sum_{i = 0}^{n-1}(\chi(Z,i^*M\otimes L^{\otimes i})- \chi(Z,\Oo_{Z}^{\oplus k}L^{\otimes^i})) = n\deg_X i^*M,$$
where we applied Lemma \ref{lemma:deg_formula} to obtain the last equality.

The second assertion follows from the first, and a second application of Lemma \ref{lemma:first}:
$$\deg_Z \pi_*M = \chi(Z,\pi_*M) - \chi(Z,\Oo_Z^{kn}) = \chi(Z_n,M) - \chi(Z_n,\Oo_{Z_n}^k) + \chi(Z_n,\Oo_{Z_n}^k) - \chi(Z,\Oo_{Z}^{kn})$$ 
and we may continue to simplify this expression as
$$\deg_{Z_n}M + k\sum_{i = 0}^{n-1}\deg_Z L = n\deg_Z i^*M + k\frac{n(n-1)}{2}\deg_Z L.$$
This concludes the proof.
\end{proof}

\begin{lemma}\label{lemma:Split_surjective}
Consider the union of connected components $\Split_d$ of $\Split$, given by splittings of fibrewise degree $d$ (in the sense of vector bundles on the spectral covers $Y^{(1)}_a$). Then, the map $\Split_d \rightarrow \A^{(1)}$ is surjective, if and only if $Y^{(1)}_0$ supports a splitting of degree $d$.
\end{lemma}

\begin{proof}
We will imitate the proof of Lemma \ref{lemma:Pic_surjective}. We observe as above that the structure map $\alpha\colon \Split \rightarrow \A^{(1)}$ is flat and locally of finite presentation, therefore it is open. As a next step we have to show that the image $\alpha(\Split_d) \subset \A^{(1)}$ is a $\G_m$-invariant open subset. The assertion follows then, since $0 \in \alpha(\Split_d)$ implies that $\alpha(\Split_d) = \A^{(1)}$, as consequence of the $\G_m$-invariance.

Unlike the previous case, $\G_m$-invariance is not automatic in our case, since there is no reason, why $\Split$ should inherit an action by $\G_m$. Instead we will show that for a geometric point $a \in \alpha(\Split_d)(k)$, we also have an inclusion of the entire orbit $\G_m(k) \cdot a \subset \alpha(\Split_d)$.

Indeed, pulling back the family of spectral curves $Y^{(1)} \rightarrow \A^{(1)}$ via the map $\G_m \rightarrow \G_m \cdot a \rightarrow \A^{(1)}$, we obtain a trivial family of curves on $\G_m$ (by means of the $\G_m$-action on $Y^{(1)}$). However, we point out that the isomorphism class of the Azumaya algebra $\D$ is not necessarily constant on this family.

If $b \in \G_m(k) \cdot a$, and $Y_b^{(1)} \cong Y^{(1)}_a$ supports a splitting $M$ of degree $c$, then every other splitting is of the form $M \otimes L$, where $L$ is a line bundle on $Y^{(1)}_a$. Therefore, the possible degree of a splitting $M$ on $Y_b^{(1)}$ is an integer of the form $kd + c$, where $k\in \mathbb{N}$ depends only on $Y^{(1)}_a$. We have seen in Corollary \ref{cor:Pic_d_surjective} that for every spectral curve $Y_b^{(1)}$ we have a subgroup $n\mathbb{Z} \subset \pi_0(\Pic(Y_b^{(1)}))$. This implies that $k|n$, in particular $k \neq 0$.

Therefore we obtain for every $b \in \G_m(k)$ an arithmetic progression $\{kd + c | d \in \mathbb{Z}\}$, which is equal to the set $\{\deg_{Y_b^{(1)}} S|S \in \Split(\D/b)\}$, and we have to show that this subset is independent of the geometric point $b$. Let $\{kd + e | d \in \mathbb{Z}\}$ be this arithmetic progression for the point $a$.

By smoothness of the map $\Split \rightarrow \A^{(1)}$, every splitting can be extended to an \'etale neighbourhood. We choose such neighbourhoods and extensions of splittings for $Y_a^{(1)}$ and $Y_b^{(1)}$. For any geometric point $p \in \G_m(k) \cdot $ factoring through the fibre product of the two neighbourhoods, we thus see that $p$ supports an extension of degree $c$ and $e$. This implies $c = kd + e$, and thus finishes the proof.
\end{proof}

\begin{defi}\label{defi:Split0}
We define $\Split^0(\D/\A^{(1)}) = \Split_{(1-p)(1-h)n}(\D/\A^{(1)})$, and refer to sections thereof as \emph{good splittings}.
\end{defi}

Lemma \ref{lemma:Split_surjective} applies to show that the map $\Split^0 \rightarrow \A^{(1)}$ is surjective. The details are contained in the proof of the following corollary.

\begin{cor}\label{cor:Split0}
The morphism $\Split^0(\D/\A^{(1)}) \rightarrow \A^{(1)}$ is surjective.
\end{cor}

\begin{proof}
As we have seen in Lemma \ref{lemma:Split_surjective}, this assertion is equivalent to showing that the spectral cover $Y_0^{(1)}$ admits a splitting of $\D$ of degree $(1-p)(1-h)n$. 

We now choose a splitting $S$ of $\D$ on $Y^{(1)}_0$ which restricts to the splitting $S_0$ corresponding to the trivial connection $(\Oo_X,d)$ of rank $1$. We denote the flat connection corresponding to $S$ by $(E,\nabla)$. We claim that $\deg_X E = p(1-h)n(n-1)$, which will conclude the proof of the assertion, because with respect to the chosen splitting $S$, $(E,\nabla)$ corresponds to the Higgs bundle $\pi_*\Oo_{Y^{(1)}_0}$ considered earlier, of degree $(1-h)n(n-1)$.

We can apply Lemma \ref{lemma:second} to compute $\deg_{Y_0^{(1)}}S = n\deg_{X^{(1)}}S_0$, where $S_0$ denotes the splitting on $X^{(1)}$ induced by the tautological connection $(\Oo_X,d)$. By Lemma \ref{lemma:deg} we have $\deg_{X^{(1)}}S_0 = (1-p)(1-h)$, which yields $\deg_{Y_0^{(1)}}S = n(1-p)(1-h)$.
\end{proof}

Next we show that splittings of $\D$ which belong to $\Split^0$ preserve (semi)stability.

\begin{lemma}\label{lemma:p}
Let $U \to \A^{(1)}$ be an \'etale morphism and assume that there is a splitting $S \in \Split^0(\D/U)$. Then the induced morphism $$\delta_S^{-1}\colon U \times_{\A^{(1)}} \Higgs(X^{(1)}) \to U \times_{\A^{(1)}} \Loc(X)$$ changes the degree of the Higgs bundles, and Higgs subbundles, by multiplication with $p$. In particular, it preserves the notion of (semi)stability.
\end{lemma}

\begin{proof}
To see that the degree of the underlying bundles is changed by multiplication with $p$ we will apply Lemma \ref{lemma:deg}, which implies that the degree change is of the form $d \mapsto pd + c$, where $c$ is an integer, depending only on $\deg_{Y^{(1)}_a} S$, where $S \in \Split(Y_a^{(1)})$ is a splitting. By Lemma \ref{lemma:deg} we have $c= -(1-p)(1-h)n + \deg_{Y_a^{(1)}} S$. Hence, the first assertions is equivalent to $\deg_{Y_a^{(1)}} S= (1-p)(1-h)n$, which agrees with the definition of $\Split^0$.

The more general case concerning the degrees of Higgs subbundles requires a different analysis. For $n = k + l$ we consider the morphism $\phi\colon \mathcal{B} = \A_k \times \A_l \to \A_n$ given by polynomial multiplication. By pullback along the canonical projections to $\A_l$ and $\A_k$ we obtain two families of spectral covers of $X$ parametrised by $\mathcal{B}$, which we denote by $Y(l)$ and $Y(k)$. This is motivated by the fact that a Higgs subbundle gives rise to a factorisation of the characteristic polynomial of the Higgs field. We have a natural inclusion map $i\colon Y(k) \rightarrow Y$.

Let $U \rightarrow \A^{(1)}$ be an \'etale morphism, such that we have a good splitting $S\in \Split^0(\D/U)$. We will show that $i^*S$ is a good splitting of $\D$ on $Y(k)^{(1)}$, that is, we will show that the fibrewise degree satisfies $\deg_{Y(k)_a^{(1)}}i^*S =(1-p)(1-h)k$. If the image of $U$ contains $0 \in \A^{(1)}$, we may use Remark \ref{rmk:deg}, that is, the fact that degree is locally constant, to reduce this assertion to the zero fibre $\deg_{Y(k)^{(1)}_0} i^*S = (1-p)(1-h)k$.

By Lemma \ref{lemma:deg} this is equivalent to the assertion that $\delta_S^{-1}$ sends an arbitrary degree $0$ rank $\ell$ Higgs bundle $(F,\theta) \in \invHDol(0)$ to a degree $0$ rank $\ell$ local system $(E,\nabla)$. 

We may choose $S$ to be a good splitting over the spectral cover $Y_0^{(1)}$. Lemma \ref{lemma:second} implies that if $j\colon X^{(1)} \rightarrow Y^{(1)}_0$ denotes the zero section, then $j^*S$ is the splitting on $X^{(1)}$, then $j^*S$ is a good splitting for the rank $1$ spectral curve $X^{(1)} \hookrightarrow T^*X^{(1)}$. 

In particular we see that for every integer $k$, the local system $(\Oo_X^{\oplus k},d)$ satisfies $0=\deg_{X^{(1)}} \delta_S^{-1}(\Oo_X^{\oplus k},d) = \deg_{X}(\Oo_{X^{(1)}}^{\oplus k},0)$. 

If $0 \in \A^{(1)}$ is not in the image of $U \rightarrow \A^{(1)}$, we argue as follows. Let $V \rightarrow \A^{(1)}$ be an \'etale neighbourhood of $0$, such that $\Split^0(\D/U)$ admits a section (that is a splitting of $\D$ over $Y^{(1)} \times_{\A^{(1)}} U$). Since the images of $U$ and $V$ are dense open subsets of $\A^{(1)}$, there exists a geometric point $x \in \A^{(1)}(k)$, which factors through $U \times_{\A^{(1)}} V$. Using one more time that the fibrewise degree function is locally constant in flat families, we deduce that $\deg_{Y(k)^{(1)_a}}i^*S =(1-p)(1-h)k$ is satisfied for every geometric point of $V$, hence also $x$, and therefore every geometric point of $U$.
\end{proof}

\begin{cor}\label{cor:proper}
The twisted Hitchin map $\HdR\colon \ssLoc \to \A^{(1)}$ is universally closed. In particular, if $\cLoc$ denotes the coarse moduli space of local systems, then $\HdR\colon \cLoc \rightarrow \A^{(1)}$ is proper.
\end{cor}

\begin{proof}
It has been shown by Nitsure (cf. \cite[Thm. 6.1]{MR1085642}) that the Hitchin map $$\HDol\colon \ssHiggs(X^{(1)}) \to \A^{(1)}$$ is universally closed. Faithfully flat descent theory allows one to deduce the analogous assertion for $\HdR$.
\end{proof} 

Corollary \ref{cor:proper} is a generalisation of a result of Laszlo and Pauly. In their paper \cite{MR1810690} they prove that the zero fibre of the twisted Hitchin fibration from the stack of semistable $t$-connections to the Hitchin base is universally closed.

\section{Geometric Langlands in positive characteristic}\label{GL}

For a smooth proper curve $X/\mathbb{C}$, the geometric Langlands correspondence refers to a conjectured equivalence of categories
$$ D^?(\Loc,\Oo) \cong D^?(\Bun,D_{\Bun}), $$
respecting various extra structures. The precise definition of the derived categories above is given in \cite{MR3300415}. The right hand side $D^?(\Bun,D_{\Bun})$ is endowed with a family of functors, called Hecke operators. The geometric Langlands correspondence is expected to intertwine those with certain multiplication operators acting on $D^?(\Loc,\Oo)$.

The picture described above is reminiscent of number theory. The Langlands programme is a collection of theorems and conjectures, encompassing a far-reaching generalisation of class field theory. Using Grothendieck's \emph{function-sheaf} dictionary it is possible to relate these two programmes. We refer to the survey \cite{MR2290768} and references therein for an account of the classical and geometric Langlands conjecture.

In their paper \cite{Bezrukavnikov:fr} Bezrukavnikov--Braverman could establish a derived equivalence over the locus of smooth spectral curves, relating $D$-modules on $\Bun$ in positive characteristic to coherent sheaves on $\Loc$. In this section we extend their results to the locus of integral spectral curves by presenting the proof of Corollary \ref{cor:Langlands}. We construct an equivalence of derived categories $D^b_{\mathrm{coh}}(\intLoc,\Oo) \cong D^b_{\mathrm{coh}}(\intHiggs,\D_{\Bun})$, where $\D_{\Bun}$ denotes the Azumaya algebra of differential operators on $\Bun$ defined on an open dense substack of $\Higgs(X^{(1)})$ (see section 3.13 in \cite{Bezrukavnikov:fr}). This equivalence is constructed using the Arinkin-Poincar\'e sheaf (\cite{Arinkin:2010uq}). For this reason we assume from now on that the characteristic of the base field $p$ satisfies the estimate
$$ p > 2n^2(h-1) + 1, $$
where $h$ denotes the genus of $X$. This assumption is required to apply Arinkin's results, as we will explain in Theorem \ref{thm:parinkin} and Lemma \ref{lemma:genus}. Moreover this equivalence can be shown to intertwine the Hecke operator with a multiplication operator. 

Corollary \ref{cor:Langlands} emphasises the similarity of autoduality phenomena of the Hitchin system (the classical limit) with the geometric Langlands programme. It provides another example for the philosophy that \emph{quantum} statements become \emph{almost classical} when specialised to positive characteristic. 

From now on abandon the context of orbicurves, and will assume that $X$ is a smooth projective curve defined over $k$. We do not know if the proofs of Theorem \ref{thm:Hecke} and Lemma \ref{lemma:planar} which rely on \cite{MR1658220} generalise straight-forwardly to the orbicurve setting.

\subsection{Splitting of $\D_{\Bun}$ on smooth Hitchin fibres}\label{splitting}

This subsection is entirely expository. We begin by reviewing the theory of abelian group stacks and refer the reader to \cite{AriApp} for a more detailed treatment. In the following we fix a base scheme $S$ and consider an abelian group $S$-stack $Y$. The dual $Y^{\dual}$ is defined to be the stack classifying morphisms of group stacks $Y \to B\G_m$. 

\begin{defi}\label{defi:dual}
Let $Y$ be an abelian group $S$-stack. We define $Y^{\dual}$ to be the group stack, which sends every $S$-scheme $T$ to the groupoid of morphisms of group stacks $Y \times_S T \to B\G_m$.
\end{defi}

If $A$ is an abelian $S$-scheme, $A^{\dual}$ is the dual abelian $S$-scheme, which can be extracted from \cite[p. 184]{MR918564}. Moreover we have $B\G_m^{\dual} = \mathbb{Z}$ and $\mathbb{Z}^{\dual} = B\G_m$. The dual group scheme functor can also be thought of as {the first derived functor of the Cartier dual}.

\begin{defi}\label{defi:Cartier}
The \emph{Cartier dual} $Y^*$ of an abelian group stack $Y$ is the stack of group stack morphisms $Y \to \G_m$. 
\end{defi}
 
If $\Gamma$ is a finite group scheme, its dual is given by the classifying stack $B\Gamma^*$ of the Cartier dual $\Gamma^*$. 

A group stack $Y$ is said to be \emph{nice}, if the natural morphism $Y \to Y^{\dual\dual}$ is an equivalence. All the examples considered above are nice.

Dualizing is in general not an exact operation, as the following counter-example shows. An isogeny of abelian varieties $A \to B$ gives rise to an exact sequence of nice group stacks
$$ 0 \to \Gamma \to A \to B \to 0, $$
the dual sequence is $B^{\dual} \to A^{\dual} \to B\Gamma^* \to 0$, but the first arrow is the dual isogeny and certainly has a non-trivial kernel in general.

The theory of abelian group stacks is used by Bezrukavnikov--Braverman to show the following Theorem (\cite[Thm. 4.10(1)]{Bezrukavnikov:fr}):

\begin{thm}
The Azumaya algebra of differential operators $\D_{\Bun}$ (see section 3.13 in \cite{Bezrukavnikov:fr}) carries a natural group structure over the locus of smooth spectral curves $\smA$. In particular we have an extension of group $\smA$-stacks 
$$ 0 \to B\G_m \to \Y_{\D_{\Bun}} \to \Pic(Y^{(1)}/\smA) \to 0.  $$
Here we use $\Y_{\D_{\Bun}}$ to denote the gerbe associated to $\D_{\Bun}$.
\end{thm}

We refer the reader to \cite[App. 5.5]{MR2373230} for a precise definition of group structures on Azumaya algebras. For our purpose it is sufficient to know that we have an extension of group stacks as stated in the theorem above. 

Using this theorem, Bezrukavnikov--Braverman conclude in \cite[Thm. 4.10(2)]{Bezrukavnikov:fr} that \'etale locally over the Hitchin base $\smA$ the Azumaya algebra $\D_{\Bun}$ splits. By virtue of the Abel-Jacobi map $$Y^{sm,(1)} \to \Pic(Y^{sm,(1)}/\smA),$$ sending $y \in Y_Y$ to the line bundle $\Oo_{Y_a}(-y)$, they compare splittings of $\D_{\Bun}$ respecting the group structure on the Hitchin fibres, with splittings of $\D_X$ on the spectral curve. 

\begin{cor}[Bezrukavnikov--Braverman]\label{cor:smthsplit}
The pullback of $\D_{\Bun}$ along the Abel-Jacobi map is canonically Morita equivalent to $\D_X$. This implies the existence of a natural isomorphism
$$ \Split_{\mathrm{grp}}(\D_{\Bun}/\smA) \cong \Split(\D_X/\smA), $$
where $\Split_{\mathrm{grp}}$ refers to the stack of relative splittings respecting the group structure. In particular, the Azumaya algebra $\D_{\Bun}$ splits \'etale locally over the base $\smA$.  
\end{cor}

We refer the reader to \cite[App. 5.5]{MR2373230} for the notion of a splitting respecting the group structure. All that matters to us, is that we can relate splittings of $\D_X$ on smooth spectral curves, to splittings of $\D_{\Bun}$ on the corresponding Hitchin fibres.

\subsection{The Langlands correspondence over $\INTA$}\label{integral}

This section is devoted to the proof of Theorem \ref{thm:Langlands}, which is an extension of the main result of \cite{Bezrukavnikov:fr}. We use the notation $\intLoc  = \intLoc(X)$ and $\intHiggs = \intHiggs(X^{(1)})$ for the moduli stacks of local systems on $X$, respectively Higgs bundles on $X^{(1)}$, of rank $n$ with integral spectral curves. The corresponding open dense subset of $\A^{(1)}$ will be denoted by $\intA$. We will freely use results and notation about Fourier-Mukai transforms \cite{MR946249}. 

\begin{thm}\label{thm:Langlands}
There exists a canonical $p_2^*\D_{\Bun}$-module $\Lb$ on
$$ \intLoc \times_{\intA} \intHiggs $$
inducing an equivalence of derived categories
$$ D^b_{\mathrm{coh}}(\intLoc,\Oo) \cong D^b_{\mathrm{coh}}(\intHiggs,\D_{\Bun}). $$
\end{thm}

\subsubsection{Relation with the Arinkin-Poincar\'e sheaf}

This section contains an introduction to autoduality for compactified Jacobians. In the following we denote by $S$ a scheme of finite type over $k$, which parametrises a flat family of locally planar, integral curves $\pi\colon C \to S$. We denote by $\Picb(C/S)$ the compactified Picard stack, that is the moduli stack of torsion-free rank $1$ sheaves on $C$, and by $\Jb(C/S)$ the compactified Jacobian. The universal family of torsion free sheaves on $\Picb(C/S) \times_S C$ will be referred to as $\Qb$. Similarly, we denote the universal line bundle on $\Pic(C/S) \times_S C$ by $\Qq$. Let $(C/S)^{[d]}$ denote the relative Hilbert scheme of $d$ points. Below we define a natural morphism $(C/S)^{[d]} \to \Picb^{d}(C/S)$ to the degree $d$ component of the compactified Picard stack. The following proposition can be found in \cite[sect. 5]{MR555258}.

\begin{prop}[Abel map]\label{prop:abel}
The Abel map $$A^d\colon (C/S)^{[d]} \to \Picb^{d}(C/S)$$ sends a family of subschemes $D$ of $C/S$ to the dual of its ideal sheaf $\I_D^{\dual}$. For $d \geq 2g-1$, where $g$ denotes the arithmetic genus of the family $C/S$, the map $A$ is smooth and surjective.
\end{prop}

On a smooth curve $X$, elements of $X^{[d]}$ are usually referred to as degree $d$ effective divisors.

We conclude by stating the following properties of $\Pp$ (see \cite{Arinkin:2007fk}).

\begin{lemma}\label{lemma:poincareproperties}
The Poincar\'e line bundle $\Pp$ on $\Picb(C/S) \times_S \Pic(C/S)$ has the following properties:
\begin{itemize}
\item[(a)] Its restriction to $\Pic(C/S) \times_S \Pic(C/S)$ is symmetric under change of coordinates. In particular, $\Pp$ extends to a line bundle on $\Picb(C/S) \times_S \Pic(C/S) \cup \Pic(C/S) \times_S \Picb(C/S)$, which will be referred to by the same notation.
\item[(b)] For every $F \in \Picb(C/S)$, the total space of the corresponding line bundle $\Pp|_{\{F\} \times \Pic}$ is endowed with a natural group structure, and thus gives rise to a central extension of $\Pic$.
\end{itemize}
\end{lemma}

According to the first and last point of the lemma above, we can restrict $\Pp$ to a line bundle on $\Jb \times_S \J \cup \J \times_S \Jb$. In \cite{Arinkin:2010uq} Arinkin has constructed a maximal Cohen-Macaulay sheaf $\bar{\Pp}$ on $\Jb \times_S \Jb$, extending the Poincar\'e line bundle $\Pp$ defined above.

\begin{thm}[Arinkin]\label{thm:parinkin}
Let $k$ be an algebraically closed field of characteristic zero or $p > 2g - 1$. Let $S$ be a $k$-scheme locally of finite type and $\pi\colon C \to S$ a flat family of integral curves with planar singularities and arithmetic genus $g$. Then there exists a coherent sheaf $\Pb$ on $\Jb \times_s \Jb$, which is Cohen-Macaulay over every closed point $s \in S$ and extends the Poincar\'e line bundle $\Pp$ on $\Jb \times_S \J \cup \J \times_S \Jb$. Relative Fourier-Mukai duality with respect to the kernel $\Pb$ gives rise to an equivalence of bounded derived categories
$$ \Phi_{\Pb}\colon D^b_{\coh}(\Jb) \to D^b_{\coh}(\Jb). $$
\end{thm}

The proof of this theorem can be found in \cite{Arinkin:2010uq} for the case of characteristic zero. The reason for restricting the characteristic is that Haiman's celebrated $n!$ Theorem (cf. Remark \ref{rmk:haiman}, \cite{MR1839919}) plays a role in the construction of $\Pb$.
\begin{rmk}\label{rmk:haiman}
If $H_n$ denotes the Hilbert scheme of length $n$ points on $\mathbb{A}^2$, Haiman's result shows that a certain natural morphism $X_n \to H_n$ is finite and flat. The scheme $X_n$ is referred to as the \emph{isospectral Hilbert scheme}, and we may also state the aforementioned result of Haiman as sayiong that $X_n$ is Cohen-Macaulay. Again the original source states the relevant theorem only in the case of characteristic zero, but the proof is easily adapted to $p > n$. The only characteristic-sensitive part of Haiman's proof is the use of Maschke's theorem for the symmetric group $S_n$, which is true as long as $p > n$ (\cite[Thm. XVIII.1.2]{MR1878556}). The combinatorial backbone of Haiman's work, the Polygraph Theorem, has already been proved over $\mathbb{Z}$ in the original publication.  
\end{rmk}

In order to construct the Arinkin-Poincar\'e sheaf one needs Cohen-Macaulayness of the isospectral Hilbert scheme $X_n$ for $n = 2g - 1$. This requires $p > 2g - 1$. But since representation theory of the symmetric group is also used in the defining formula \cite[(4.1)]{Arinkin:2010uq} of the Arinkin-Poincar\'e sheaf, Arinkin's methods depend on the restriction $p > 2g -1$ a second time.

However, as shown by the recent work \cite{Melo:2013aa}, it is actually sufficient to assume $p > g$.

Over the smooth open subscheme $J \subset \Jb$, the Arinkin-Poincar\'e sheaf $\bar{\Pp}/\Jb \times \Jb$ restricts to the Poincar\'e line bundle $\Pp/\J \times \Jb$. Moreover if the generic member of the family $C \to S$ is smooth, the codimension of the complement of $J \subset \bar{\J}$ is $\geq 2$. A Cohen-Macaulay sheaf can be recovered as the push-forward of the restriction of an open subset, which is the complement of a closed subset of codimension $\geq 2$. This allows us to reconstruct the original sheaf by push-forward from the smooth locus
$$ \bar{\Pp} = i_*\Pp. $$

The following theorem is a direct consequence of \cite[Thm. C]{Arinkin:2010uq}.

\begin{thm}\label{thm:stackyarinkin}
With the same assumptions as in Theorem \ref{thm:parinkin} there exists a natural automorphism of derived categories
 $$ \Phi_{\Pb}\colon D^b_{\coh}(\Picb) \to D^b_{\coh}(\Picb) $$
 given by an integral kernel $\Pb$, which is a Cohen-Macaulay coherent sheaf on the stack $\Picb \times_S \Picb$. Moreover $\Pb$ is the (non-derived) push-forward of the Poincar\'e line bundle $\Pp$ on $$\Picb \times_S \Pic \cup \Pic \times_S \Picb.$$
\end{thm}

\begin{proof}
According to Lemma \ref{lemma:poincareproperties}, a section of the degree one component $\pi\colon \Pic^1(C/S) \to S$ induces not only a splitting of the extension $$0 \to J \to \Pic^{rig} \to \mathbb{Z} \to ,0$$ but also of the group gerbe $\Picb \to \Picb^{rig}$. Since $\pi$ is a smooth map of stacks, such a section can be chosen \'etale locally. This yields an \'etale local identification of $\Picb$ with $\mathbb{Z} \times \Jb \times B\G_m$, under which $\Pp$ just restricts componentwise to the Poincar\'e line bundle on $J \times_S J_b \cup \Jb \times_S J$. In particular, we see that $i_*\Pp$ is a maximal Cohen-Macaulay sheaf on $\Picb \times_S \Picb$ and induces an equivalence of derived categories.
\end{proof}

\subsubsection{Twisting Arinkin's equivalence}

We denote by $\intA \subset \A^{(1)}$ the open subset corresponding to integral spectral curves $Y_a^{(1)}$.

We would like to construct an integral kernel, that is, a sheaf $\Lb$ on the fibre product $\intLoc \times_{\A^{(1)}} \intHiggs$, which is endowed with the structure of a $p_2^* \D_{\Bun}$-module. As soon as this is done, we are able to set-up a Fourier-Mukai functor between derived categories, analogous to \cite{MR946249}.

Let us denote by $$\phi\colon \intSplit \times_{\intA} \intLoc \to \intSplit \times_{\intA} \intHiggs$$ the restriction of the isomorphism of Theorem \ref{thm:main} to the integral locus $\intA$. According to Corollary \ref{cor:smthsplit} there exists a $p_2^* \D_{\Bun}$-splitting $\Ll$ on $\Split \times_{\smA} \smHiggs$. Since the total space of $\Split \times_{\intA} \intHiggs$ is smooth we can extend $\Ll$ to a splitting on $\Split \times_{\intA} \Picb$ (Corollary 2.6 in \cite{MR559531}). On every connected component this extension is unique up to tensoring by a line bundle $\chi^*L$ pulled back from $\intA$. But $\Pic(\intA) = 1$, since $\intA \subset \A^{(1)}$ is an open subscheme of affine space. We choose one and denote it again by $\Ll$. 

\begin{lemma}\label{lemma:L-local}
We have $\Hhom_{p_3^*\D_{\Bun}}(p_{13}^*\Ll,p_{23}^*\Ll) = \phi^*p_{23}^*\Pp$ on
$$ \intSplit \times_{\intA} \intLoc \times_{\intA} \intHiggs \cong \intSplit \times_{\intA} \intHiggs \times_{\intA} \intHiggs. $$
\end{lemma}

\begin{proof}
All constituents in the above identity are either splittings or line bundles, which arise as pullbacks of some extension of the same object defined over $\smA$. In particular they are well-defined up to twisting by the same line bundle defined on $\A^{(1)}$. We may therefore conclude that it suffices to check the identity over $\smA$. Over the locus of smooth spectral curves it holds by means of the Abel-Jacobi map defining the families of line bundles, respectively splittings, given by $\Pp$ and $\Ll$ (Corollary \ref{cor:smthsplit}).
\end{proof}

We denote by $$j\colon \Split \times_{\A^{(1)}} \intHiggs \to \intHiggs \times_{\A^{(1)}} \intHiggs$$ the inclusion of this open substack. We observe that its complement has codimension $\geq 2$, since $\chi_{dR}$ is flat and $j$ is surjective over the locus of smooth spectral curves $\smA$. Theorem \ref{thm:stackyarinkin} motivates the next definition.

\begin{defi}\label{defi:kernel}
We define the  $p_2^* \D_{\Bun}$-module $\Lb$ to be $j_*\Ll$.
\end{defi}

Using this identification we obtain the identity $$ (\phi \times_{\intA} \id_{\intHiggs})^*\Hhom_{p_{3}^*\D_{\Bun}}(p_{13}^*\Ll,p_{23}^*\Lb) = p_{13}^*\Pb $$ on $$ \intSplit \times_{\intA} \intLoc \times_{\intA} \intHiggs \cong \intSplit \times_{\intA} \intHiggs \times_{\intA} \intHiggs. $$

\'Etale locally on $\intA$ we can choose a section of $\intSplit \to \intA$. Using descent theory, we conclude the following lemmas.

\begin{lemma}\label{lemma:CM}
The sheaf $\Lb$ is Cohen-Macaulay.
\end{lemma}

\begin{proof}
As we have seen above it is possible to choose a splitting of $\D$ \'etale locally on $\intA$, which sends $\Lb$ to $\Pb$. Since the latter sheaf is Cohen-Macaulay and it differs from $\Lb$ only by tensoring with a locally free sheaf, we conclude that \'etale locally $\Lb$ is Cohen-Macaulay too. According to Corollary 2.1.8 of \cite{MR1251956}, the coherent sheaf $\Lb$ is Cohen-Macaulay if and only if its restriction to the formal neighbourhood of each closed point is Cohen-Macaulay. But since $k$ is assumed to be algebraically closed, there are no non-trivial \'etale coverings of formal neighbourhoods of closed points (\cite[Prop. 4.4]{MR559531}). We may therefore conclude that $\Lb$ is Cohen-Macaulay when restricted to the formal fibres of $\intLoc \times_{\intA}\intLoc \to \intA$, which allows us to conclude that it is Cohen-Macaulay itself.
\end{proof}

\begin{lemma}\label{lemma:GL}
The Fourier-Mukai functor $\Phi_{\Lb}\colon D^b_{coh}(\intLoc,\Oo) \to D^b_{coh}(\intHiggs,\D_{\Bun})$ is an isomorphism.
\end{lemma}

\begin{proof}
According to Theorem C in \cite{Arinkin:2010uq} the inverse integral kernel to $\Pb$ is a complex supported in a single cohomological degree. This allows us to evoke faithfully flat descent in order to construct the inverse kernel for $\Lb$, which would require more care for a general object of the derived category. 

We denote by $\phi\colon S \to \intA$ an \'etale covering, such that $\D$ splits when pulled back to $\phi^*\intHiggs = S \times_{\intA} \intHiggs$. The existence of such a $\phi$ is guaranteed by Theorem \ref{thm:main}. In order to simplify notation we denote the base change of an  arbitrary morphism $Z \to \A^{(1)}$ along $S \to \A^{(1)}$ by $\phi^*Z \to S$. The same convention applies to objects defined over $Z$. The Fourier-Mukai transform $\Phi_{\phi^*\Lb}$ is an equivalence of categories of Fourier-Mukai type, according to the fact that $\phi^*\Lb$ can be naturally related to $\Pb$. Moreover we see that there is an integral kernel $\bar{\mathcal{K}}'$ defined over $$\phi^*\intLoc \times_{S} \phi^*\intHiggs,$$ which gives rise to the quasi-inverse equivalence $\Phi_{\phi^*\Lb}$. Since $\Phi_{\phi^*\Lb}$ is a Fourier-Mukai transform corresponding to a kernel pulled back along $\phi$, it is naturally equipped with descent data along the morphism $\phi$. The quasi-inverse $\Phi_{\phi^*\Lb}^{-1}$ is therefore endowed with the same descent data, which its integral kernel $\bar{\mathcal{K}}'$ inherits. As we have already pointed out above, Theorem C in \cite{Arinkin:2010uq} implies that $\bar{\mathcal{K}}'$ is a coherent sheaf up to shift, which allows us to descend it to a complex $\bar{\mathcal{K}}$ on $$\intLoc \times_{\intA} \intHiggs$$ (which is again just a coherent sheaf up to shift) by faithfully flat descent. We conclude that $\Phi_{\Lb}$ and $\Phi_{\bar{\mathcal{K}}}$ are inverse to each other, since the cohomological computations involved in checking this can be verified \'etale locally on the base according to the flat base change theorem.
\end{proof}

In order to explain the restriction we have to put on the characteristic of the base field, we need to calculate the arithmetic genus of spectral curves. From the characteristic zero theory one expects the arithmetic genus of spectral curves to be $n^2(h-1) + 1$, where $n$ denotes the rank of the Higgs bundles respectively local systems, since the genus of a smooth spectral curve equals the dimension of its Picard, that is, the corresponding Hitchin fibre. Due to the Lagrangian property of a Hitchin fibre, this dimension is half the dimension of the total space, that is, the same as the dimension of the moduli space of vector bundles $n^2(h-1) + 1$. For general fields we arrive at the same number by a simple Riemann-Roch computation.

\begin{lemma}\label{lemma:genus}
The arithmetic genus $g$ of a spectral curve $Y_a$ of a curve $X$ of genus $h$ and a Higgs bundle of rank $n$ is given by
$g = n^2(h-1) + 1$ 
\end{lemma}

\begin{proof}
Because the arithmetic genus is constant in flat families it suffices to calculate the genus of smooth spectral curves. Let $\pi\colon Y_a \to X$ denote the finite morphism of a smooth spectral curve to $X$. Since $\pi$ is finite, we know that $\chi(\pi_*\Oo_{Y_a}) = \chi(\Oo_{Y_a})$. The right hand side is given by $1 - g$ according to the Riemann-Roch formula. The left hand side is constant in flat families and thus may be computed for a particular spectral curve. If $\Theta_X$ denotes the sheaf of tangent vector fields on $X$, then $$ \pi_*\Oo_{Y_a} = \bigoplus_{i=0}^{n-1}\Theta_X^{\otimes i}. $$
Combining this with the Riemann-Roch formula we compute $$ \chi(\pi_*\Oo_{Y_0}) = \sum_{i=0}^{n-1}(1-h + 2i(1-h)) = n^2(1-h).$$
In particular we obtain that the arithmetic genus of a spectral curve is given by $n^2(h-1)+1.$
\end{proof}

This concludes the proof of Theorem \ref{thm:Langlands}. We finish this section by an important remark, which is proved by exactly the same methods as Theorem \ref{thm:Langlands}.

The proposition below is a relativisation of Lemma \ref{lemma:GL}. The case $S=X$ is of particular importance, since it allows us to formulate the Hecke eigenproperty.

\begin{prop}\label{prop:relative}
Let $S$ be a smooth $k$-scheme locally of finite type. Then we have an equivalence of derived categories
$$ D^b_{coh}(\intHiggs \times T^*S^{(1)},\D_{\Bun \times S}) \cong D^b_{\coh}(\intLoc \times T^*S^{(1)}, \Oo_{\Loc} \boxtimes \D_S), $$
which is induced by the pullback of $\Lb$. 
\end{prop}

\begin{proof}
Let us denote by $\phi\colon \intA \times S \to \intA$ the canonical projection to the first component. Since $S \to \Spec k$ is faithfully flat, the same holds for the base change $\phi$. As in the proof of Lemma \ref{lemma:GL} we denote by $\phi^*Z$ the base change of the $\intA$-scheme $Z$ along the map $\phi$. The $\Oo_{\phi^*\intLoc} \boxtimes \phi^*\D_{\Bun}$-module $\phi^*\Lb$ induces a functor $$D^b_{coh}(\intHiggs \times T^*S^{(1)},\D_{\Bun \times S}) \to D^b_{\coh}(\intLoc \times T^*S^{(1)}, \Oo_{\Loc} \boxtimes \D_S).$$ Using the descent argument of the proof of Lemma \ref{lemma:GL} we conclude that this Fourier-Mukai transform is an equivalence.
\end{proof}

\subsection{The Hecke eigenproperty}\label{hecke}

The equivalence of Theorem \ref{thm:Langlands} can be shown to intertwine two canonical functors with each other. This is expected from the geometric Langlands conjecture over $\mathbb{C}$.

Let $\Ee$ denote the universal vector bundle over $\Loc \times X$. It gives rise to a multiplication functor.

\begin{eqnarray*}
\Wb\colon &D^b_{\mathrm{coh}}(\Loc, \Oo) \to D^b_{\mathrm{coh}}(\Loc \times X, \Oo \boxtimes \D_X)\\
 &M \mapsto  Lp_1^*M \otimes \Ee
\end{eqnarray*}

We define the stack $\Hh$ to be the classifying stack of the data $(E,F,\iota,x)$, such that $E,F \in \Bun$, $\iota\colon E \to F$ is an injection, $x \in X$ and $\coker \iota$ is a skyscraper sheaf of length one.

Note that $\Hh$ is equipped with two natural morphisms
\[
\xymatrix{& \Hh \ar[ld]_q \ar[rd]^p \\
\Bun & & \Bun \times X}
\]
sending $(E,F,\iota,x) \mapsto F$ respectively $(E,F,\iota,x) \mapsto (E,x)$. 

The following remark will be of use later to motivate the definition of Hecke operators in positive characteristic.

\begin{rmk}\label{rmk:Hecke}
The stack of Hecke operators $\Hh$ is actually a moduli stack for a certain type of (quasi-)parabolic bundles. Therefore the corresponding moduli stack of (quasi-)parabolic Higgs bundles is equivalent to the cotangent stack $T^*\Hh$. 
\end{rmk}

We define the Hecke operator $\Hb$ to be the functor:
\begin{eqnarray*}
\Hb\colon &D^b_{\mathrm{coh}}(\Bun, D) \to D^b_{\mathrm{coh}}(\Bun \times X, D)\\
 &M \mapsto  LRp_*LRq^!M.
\end{eqnarray*}
Whereas the definition of $\Wb$ makes immediately sense for the smaller stack $\intLoc$, it is not obvious that $\Hb$ descends to a functor
$$D^b_{\mathrm{coh}}(\intHiggs, \D) \to D^b_{\mathrm{coh}}(\intHiggs \times X, \D_{\Bun \times X}).$$

In order to see that this is the case we need to remind the reader of the definition of the functors $p_*$ and $q^!$, respectively their derived versions $LRp_*$ and $LRq^!$, as defined in \cite{Bezrukavnikov:fr}.

\subsubsection{A reformulation of the Hecke operator}

In order to define the functor $q^!\colon \D_{\Bun}-\Mod \to \D_{\Hh}-\Mod$ we have to consider the morphism $$dq^{(1)}\colon q^{(1),*}T^*\Bun^{(1)} \to T^*\Hh^{(1)},$$ and use that $q^{(1),*}\D_{\Bun}$ and $dq^{(1),*}\D_{\Hh}$ are canonically Morita equivalent \cite[Prop. 3.7]{Bezrukavnikov:fr}. 

Analogously we need to consider $$dp^{(1)}\colon p^{(1),*}T^*(\Bun \times X)^{(1)} \to T^*\Hh^{(1)},$$ and the Morita equivalence of $dp^{(1),*}\D_{\Hh}$ with $p^{(1),*}\D_{\Bun \times X}$ to define $$p_*\colon \D_{\Hh}-\Mod \to \D_{\Bun \times X}-\Mod.$$

The most natural way to deal with those two morphisms simultaneously is to look at their fibre product 

\[
\xymatrix{ Z \ar[r] \ar[d] & q^{*}T^*\Bun \ar[d]\\
 p^{*}T^*(\Bun \times X) \ar[r] & T^*\Hh.}
\]

The stack $Z$ is the domain of the morphisms\footnote{Note that we use the notation $p$ and $q$ to denote morphisms which are really base changes thereof.} $$\alpha_1 = q \circ \pr_1\colon Z \to \Higgs(X)$$ and $$\alpha_2 = p \circ \pr_2\colon Z \to \Higgs(X) \times T^*X.$$ On $Z^{(1)}$ we then have three natural Azumaya algebras, $\alpha_1^*\D_{\Bun}$, $\alpha_2^*\D_{\Bun \times X}$ and $\pi^*\D_{\Hh}$, where $\pi\colon Z^{(1)} \to \Hh^{(1)}$ denotes the structural morphism of the fibre product $Z^{(1)}$. By construction, all these algebras are pairwise Morita equivalent \cite[Prop. 3.7]{Bezrukavnikov:fr}.

The base change formula implies that $$\Hb\colon M \mapsto R\alpha_{2,*}L\alpha_1^*M,$$ where the application of Morita equivalences (which involves tensoring by a splitting) has been suppressed to simplify notation. We will later turn this into a definition of $\Hb$. Let us record the following observation of \cite[4.16]{Bezrukavnikov:fr}.
\begin{lemma}\label{lemma:Z}
The stack $Z$ is given by the lax 2-functor sending an $\A$-scheme $S \to \A$ to the groupoid classifying $\{x\colon S \to X \times \A, L_1 \subset L_2\}$, such that $\pi_*L_1, \pi_*L_2 \in \Higgs$ and $x^*(L_2/L_1)$ is locally free of rank $1$. Here we make use of the BNR correspondence (Theorem \ref{thm:bnrHiggs}) to describe Higgs bundles in terms of sheaves on the spectral curves.
\end{lemma}
\begin{proof}
We prove this lemma on the level of $k$-points and leave the only notationally different case of $S$-families to the reader. According to remark \ref{rmk:Hecke} the stack $T^*\Hh$ classifies the data $$(E,F,\theta,x,\xi),$$ where $(x,\xi) \in T^*X^{(1)}$, $(E,F,x) \in \Hh$ and $$\theta\colon F \to F \otimes \OmegaXone(x),$$ such that $\res \theta$ is a nilpotent endomorphism of the fibre $F \otimes k_x$ factoring through a linear map $$ F/E \to E \otimes k_x. $$
The morphism $T^*(\Bun \times X) \times_{\Bun \times X}\Hh \to T^*\Hh$
is given by $$[(F,\theta,x,\xi),(E,F,x)] \mapsto (E,F,\theta,x,\xi).$$
Note that $\res \theta = 0$ in this particular case, since $\theta$ is a non-singular Higgs field on $F$. The morphism $T^*\Bun \times_{\Bun} \Hh \to T^*\Hh$ can be described as $$ [(E,\theta),(E,F)] \mapsto (E,F,\theta',x,\xi), $$
where we use that $E|_{X - \{x\}} \cong F|_{X - \{x\}}$ and therefore the Higgs field $\theta$ on $E$ induces a Higgs field $\theta'$, possibly having a simple pole at $x$ on $F$. By construction this is a (quasi-)parabolic Higgs bundle. The $1$-form $\xi$ at $x$ is the eigenvalue of the Higgs field $\theta'$ on the length one quotient $F/E$. Note that this is a sensible definition since $\theta'$ preserves $E$ by construction, and that $\res(\theta)\colon E/F \to E/F$ is the zero map according to the axioms of parabolic Higgs bundles.

Computing the base change $Z$ now, with this information at hand, we obtain that $Z$ classifies 
$$ (E,F,\theta,x,\xi), $$
where $(F,\theta)$ is a Higgs bundle, $E \subset F$ is preserved by $\theta$ and $F/E$ is a length one sheaf acted on by $\theta$ with eigenvalue $\xi$.
\end{proof}

Finally we obtain a definition of $\Hb$ which can be used in our context. We observe that the two morphisms to the Hitchin base $Z \to \A$, given by $\chi \circ \alpha_1$ and $\chi \circ \pr_1 \circ \alpha_2$ agree. This is a consequence of Lemma \ref{lemma:Z}, as a point of $Z$ consists of two Higgs bundles identified away from a point $x$. In particular they have the same characteristic polynomial. This allows us to view $Z^{(1)}$ as $\A^{(1)}$-stack, and in particular to form the base change over the integral locus. 

Using $Z^{(1)}$ as a correspondence, we obtain a functor $$ \Hb\colon D^b_{coh}(\intHiggs,\D_{Bun}) \to D^b_{coh}(\intHiggs \times T^*X^{(1)},\D_{Bun \times X}). $$
Interpreting the category on the left-hand-side as a derived category of $D$-modules on $\Bun$ supported on the integral locus, and the right-hand-side as an analogous category of $D$-modules on $\Bun \times X$, we can be satisfied with $\Hb$ as a positive characteristic analogue of Hecke functors. The remainder of this paper is devoted to the proof of the following theorem, which is a formal consequence of Theorem \ref{thm:Hecke} below.

\begin{thm}\label{thm:HW}
The equivalence of Theorem \ref{thm:Langlands} intertwines $\Hb$ with $\Wb$, that is, it gives rise to the following 2-commutative diagram of derived categories
\[
\xymatrix{ D^b_{\coh}(\intLoc, \Oo) \ar[r] \ar[d] & D^b_{\coh}(\intLoc \times T^*X^{(1)}, \Oo_{\Loc} \boxtimes \D_X ) \ar[d] \\ 
D^b_{\coh}(\intHiggs,\D_{\Bun}) \ar[r] &  D^b_{\coh}(\intHiggs \times T^*X^{(1)}, \D_{\Bun \times X}).}
\]
\end{thm}

\subsubsection{The Hecke eigenproperty of the Arinkin-Poincar\'e sheaf}\label{sub:arinkinhecke}
In order to establish that the equivalence of Theorem \ref{thm:Langlands} intertwines the Hecke operator $\Hb$ with a multiplication operator, we will show that the Arinkin-Poincar\'e sheaf $\Pb$ satisfies a similar property. Given a $k$-scheme $S$ of finite type and a flat family of integral curves with planar singularities $C \to S$ we can construct the compactified relative Picard stack $\Picb \to S$, which classifies flat families of rank one torsion free sheaves $L$ on $C/S$ (Section \ref{bnr}). Let us denote the universal family  on $\Picb \times_S C$ by $\Qb$. The fibre product $\Picb \times_S \Picb$ is endowed with a Cohen-Macaulay sheaf $\Pb_{C/S}$, which induces an equivalence $\Phi_{\Pb_{C/S}}\colon D^b_{\coh}(\Picb) \to D^b_{\coh}(\Picb)$ (\cite[Thm. C]{Arinkin:2010uq}, Theorem \ref{thm:stackyarinkin}).
\begin{defi}\label{defi:hecke-arinkin}
The stack $\Hh$ classifies quadruples $(L_1,L_2,\iota,x)$, such that $L_i \in \Picb$ and $\iota\colon L_1 \subset L_2$ and $\coker \iota$ is a length one coherent sheaf supported at $x$. We have a natural morphism $\Hh \to \Picb \times \Picb \times C$ and composing it with the projections $p_2$, respectively $p_{13}$ we obtain morphisms $q\colon \Hh \to \Picb$ and $p\colon \Hh \to \Picb \times_S C$. This allows us to define the \emph{Hecke functor} $$ \Hb_{C/S} = Rp_* \circ Lq^*\colon D^b_{\coh}(\Picb) \to D^b_{\coh}(\Picb \times_S C).  $$
\end{defi}
With this definition at hand, we can state the following theorem, which has been known to Arinkin.
\begin{thm}\label{thm:Hecke}
The Fourier-Mukai transform $\Phi_{\Pb}$ intertwines the Hecke functor $\Hb_{C/S}$ with the multiplication functor $- \boxtimes^L \Qb$. In other words, we have a 2-commutative diagram of categories:
\[
\xymatrix{ D^b_{\coh}(\Picb) \ar[r]^{- \boxtimes^L \Qb} \ar[d]_{\Phi_{\Pb}} & \;\; D^b_{\coh}(\Picb \times_S C) \ar[d]^{\Phi_{\Pb_{C\times_S C / C}}} \\ 
D^b_{\coh}(\Picb) \ar[r]_{\Hb_{C/S}} &  D^b_{\coh}(\Picb \times_ S C)}
\]
\end{thm}

The subsequent proof uses the strategy developed in \cite{Arinkin:2010uq}. By the virtue of the base change theorem, we replace $S$ by the moduli stack $\M_g$ of integral curves with planar singularities of arithmetic genus $g$ and study the universal family $\C \to \M_g$. In this case $\Picb$ is a smooth stack (cf. \cite[Thm B.2]{MR1658220}).

Translating Theorem \ref{thm:Hecke} to integral kernels, we see that it suffices to show the following proposition.

\begin{prop}\label{prop:Hecke}
We have
$R(p \times_{\M_g} \id_{\Picb})_*L(q \times_{\M_g} \id_{\Picb})^*\Pb = \Pb \boxtimes^L \Qb.$
\end{prop}
We denote by $\Hh^{\sm}$ the open substack of $\Hh$ given by the preimage $p^{-1}(\Pic \times C)$. By definition $\Hh^{\sm}$ classifies quadruples $(L_1,L_2,\iota,x)$, where $L_2$ is a line bundle. Note that $L_1$ is then uniquely defined through $L_2$ and $x$. It is given by $\I_x \otimes L_2$, that is, the twist of the ideal sheaf of $x$ by the line bundle $L_2$. This implies the following Lemma.
\begin{lemma}
The restriction $p|_{\Hh^{\sm}}$ is an isomorphism onto its image.
\end{lemma}
Let us now restrict Proposition \ref{prop:Hecke} to the open substack $\Pic \times_{\M_g} \Picb \times_{\M_g} C$. The Hecke functor on the left-hand-side may be replaced by $L(q|_{\Hh^{\sm}} \times \id_{\Picb})^*$, since $p$ is an isomorphism. Both sides of the identity are underived, as $\Pp$ is an invertible sheaf.
\begin{lemma}\label{lemma:smoothHecke}
Let $q\colon \Pic \times_{\M_g} \C \to \Picb$ be the morphism sending a pair $(L,x)$ consisting of a line bundle $L$ on a curve $C$, and a point $x \in C$, to the twist $L(x) = \Hhom(\I_x,L)$, where $\I_x$ denotes the ideal sheaf of $x \in X$. Then we have
$$ L(q \times_{\M_g} \id_{\Picb})^*\Pp = \Pp \boxtimes \Qq, $$
which implies Proposition \ref{prop:Hecke} when restricted to $\Pic \times_{\M_g} \Picb$, respectively\\ $\Pic \times_{\M_g} \Picb \times_{\M_g} \C$.
\end{lemma}
\begin{proof}
According to Lemma \ref{lemma:pullCM}, $$L(q \times_{\M_g} \id_{\Picb})^*\Pp \cong (q \times_{\M_g} \id_{\Picb})^*\Pp$$ is a maximal Cohen-Macaulay sheaf. Since $$\Pic \times_{\M_g} \Pic \times_{\M_g} \C^{\sm} \subset \Pic \times_{\M_g} \Picb \times_{\M_g} \C$$ has a complement of codimension $2$, it is thus sufficient to check the identity restricted to $$\Pic \times_{\M_g} \Pic \times_{\M_g} \C^{\sm}.$$ Over this locus we are able to describe $\Pp|_{\Pic \times_{\M_g} \Pic}$ as a family of $\G_m$-extensions of $\Pic$. In particular, if 
$$(m,id)\colon \Pic \times_{\M_g} \Pic \times_{\M_g} \Pic \to \Pic \times_{\M_g} \Pic $$
is given by identity in the second component and multiplication in the first, we have an isomorphism
$$ (m,id)^*\Pp \cong p_{13}^*\Pp \otimes p_{23}^*\Pp. $$
If $A\colon \C^{\sm} \to \Pic$ denotes the Abel-Jacobi map $x \mapsto \Oo_C(x)$ then $q = m \circ (id,A,id)$ and $$(id,A)^*(\Pp|_{\Pic \times_{\M_g} \Pic}) \cong \Qq|_{\Pic \times_{\M_g} \C}.$$ Therefore, we obtain $$(q\times id)^*\Pp \cong (id,A,id)^*(m,\id)^*\Pp \cong (id,A,id)^*p_{13}^*\Pp \otimes p_{23}^*\Pp \cong p_{13}^*{\Pp} p_{23}^*\Qq. $$
\end{proof}
\begin{lemma}\label{lemma:planar}
For every $C \in \M_g$ there exists a versal deformation $\C'$ along a complete local ring $U = \Spec \hat{R} \to \M_g$, such that $U \times_{\M_g} \Cc$ has a Zariski covering $\{V_i\}_{i \in I}$ and for every $i$ we have an open immersion $V_i \to U \times \mathbb{A}^2$.
\end{lemma}
\begin{proof}
This is a combination of Lemma A.2 and Proposition A.3 in \cite{MR1658220}.
\end{proof}
\begin{lemma}\label{lemma:lci}
The morphism $\Hh \to \M_g$ is syntomic\footnote{that is, locally of finite presentation, flat and fibrewise of locally complete intersection} of relative dimension $g+1$. Moreover it is fibrewise irreducible.
\end{lemma}
\begin{proof}
As we have seen in Proposition \ref{prop:abel}, for $d \geq 2g - 1$ the morphism $(\C/\M_g)^{[d]} \to \Picb^{d}(\C)$ is smooth and faithfully flat. Under the dictionary between effective divisors and torsion sheaves, provided by the Abel map, $\Hh$ corresponds to the nested Hilbert scheme $(\C/\M_g)^{n+1,n}$. It is easy to conclude that the relative dimension of $\Hh$ is $g+1$. We only need to show that for an integral curve $C$ of arithmetic genus $g$ and planar singularities, $\dim C^{[n+1,n]} = n+1$. This can be done as in the proof of Theorem 5 in \cite{MR0498546}, complementing the estimate (5.1) by the same inequality for nested Hilbert schemes (Proposition 2.3 in \cite{MR1714770}). 

We will show that $(\C/\M_g)^{n+1,n}$ is a locally complete intersection. As we have seen in Lemma \ref{lemma:planar} we can pick a versal deformation $\C' \to U$ of $C \in \M_g$ together with a Zariski covering $\bigcup_{i \in I}V_i = \C'$ and open immersions $V_i \to U \times \mathbb{A}^2$. Therefore, we obtain a cartesian square
\[
\xymatrix{ (V_i/U)^{d+1,d} \ar[r] \ar[d] & (U \times \mathbb{A}^2/U)^{[d+1,d]} \ar[d] \\
(V_i/U)^{[d+1]} \ar[r] & (U \times \mathbb{A}^2/U)^{[d+1]}.  }
\]
It is known (cf. \cite[Cor. 7]{MR0498546}) that $(V_i/U)^{[d+1]} \subset (U \times \mathbb{A}^2/U)^{[d+1]}$ is a locally complete intersection. Since its base change $(\C'/U)^{[d+1]}$ has the same codimension in $(U \times \mathbb{A}^2/U)^{[d+1,d]}$, and the latter stack is smooth\footnote{This is proved in \cite{MR1616606}, see section 0.2 for a statement of his result. Although the author assumes $k = \mathbb{C}$ this statement and its proof are true for general algebraically closed fields.}, we have shown that the total space $\Hh$ is a locally complete intersection, in particular it is Cohen-Macaulay. Because $\M_g$ is smooth (\cite[Prop. 4]{Arinkin:2007fk}) we see that $\Hh \to \M_g$ is flat, since the dimension of the fibres is constant. Applying the above argument fibrewise, we see that the fibres are locally complete intersections. Irreducibility is verified as in the proof of Theorem 5 in \cite{MR0498546}.
\end{proof}
\begin{lemma}\label{lemma:pullCM}
The complex $L(q \times \id)^*\Pb$ is a Cohen-Macaulay sheaf $(q \times \id)^*\Pb$.
\end{lemma}
\begin{proof}
Lemma 2.3 in \cite{Arinkin:2010uq} states that derived pullback along $Lf^*$ along a morphism $f\colon Y \to X$ of schemes preserves maximal Cohen-Macaulay sheaves, if $X$ is Gorenstein of pure dimension, $Y$ is Cohen-Macaulay, and $f$ is Tor-finite.

We already know that $\Hh \to \M_g$, and hence its base change $\Hh \times_{\M_g} \Picb$ are syntomic. In particular we can conclude that both spaces are locally complete intersections (that is, also Cohen-Macaulay), since $\M_g$ is smooth. Similarly one concludes that $\Picb\times_{\M_g} \Picb$ is a locally complete intersection (that is, Gorenstein). According to Theorem B.2 in \cite{MR1658220} the total space of $\Picb$ is smooth. For this reason every morphism mapping into $\Picb$ is Tor-finite. Since Tor-finite morphisms are preserved by flat base change, we may conclude that $\Hh \times_{\M_g} \Picb \to \Picb\times_{\M_g} \Picb$ is Tor-finite.
\end{proof}
The final ingredient in the proof of the Heck eigenproperty is the following characterization of Cohen-Macaulay sheaves as objects in the derived category. It is proven in \cite[Lemma 7.7]{Arinkin:2010uq}. It gives a characterization of Cohen-Macaulay sheaves on $X$ amongst the objects of the derived category $D^b_{coh}(X)$.

\begin{prop}\label{prop:CMderived}
Let $X$ be an algebraic stack of finite typer over a field $k$, which is moreover of pure dimension. Then $\F^{\bullet} \in D^b_{coh}(X)$ is a Cohen-Macaulay sheaf of codimension $d$ if and only if the following conditions are satisfied
\begin{itemize}
\item[(a)] $\codim \supp \F^{\bullet} \geq d,$
\item[(b)] $H^i(\F^{\bullet}) = 0$ for $i > 0$,
\item[(c)] $H^i(\Db_X\F^{\bullet}) = 0$ for $i>d$.
\end{itemize}
\end{prop}
We intend to apply this result to the integral kernel $\Theta$ of the Fourier-Mukai transform $\Phi_{\Pb^{\dual}} \circ \Hb_{C/\M_g} \circ \Phi_{\Pb \boxtimes \Oo_{\C}}$. If we can show that it is a Cohen-Macaulay sheaf of the right codimension, then it suffices to determine it in a complement of a closed subset (of the support) of codimension 2.

\begin{lemma}\label{lemma:codim} 
We have $\codim \supp \Theta \geq g$ and every maximal-dimensional component intersects $\Delta_{\Picb/\M_g} \times_{\M_g} \C$.
\end{lemma}
\begin{proof}
Let $(F_1,F_2,x) \in \supp \Theta$. By base change this is the case if and only if there exists $i \in \mathbb{Z}$, such that $\mathbb{H}^i(\Hh, \Hb(\Pb_{F_1}) \otimes \Pb^{\dual}_{F_2}) \neq 0$. Similarly to the proof of Proposition 7.2 of \cite{Arinkin:2010uq} we claim that $\Hb(\Pb_{F_1}) \otimes \Pb^{\dual}_{F_2}$ is $T$-equivariant, where $T$ denotes the $\G_m$-extension of $\Pic$ associated to $\Pp_1 \otimes \Pp^{\dual}_2$.

We denote by $T_i$ the $\G_m$-extension of $\Pic$ associated to $\Pb_{F_i}$. The sheaf $\Pb_{F_i}$ has a natural $T_1$-equivariant structure (\cite[Lemma 6.5]{Arinkin:2010uq}). The morphisms $p$ and $q$ in the correspondence diagram defining $\Hb$ are $T_1$-equivariant. Consequently $\Hb(\Pb_{F_1})$ is an element of the $T_1$-equivariant derived category (that is, an object in the derived category of the stack $[(\Picb \times C_s)/T_1]$). Consequently, the tensor product $\Hb(\Pb_{F_1}) \otimes \Pb^{\dual}_{F_2}$ lies in the $T$-equivariant derived category.

The hypercohomology group $\mathbb{H}^i(\Hh, \Hb(\Pb_{F_1}) \otimes \Pb^{\dual}_{F_2}) \neq 0$ carries an induced $T$-action, such that the $\G_m$-part acts tautologically. If this group was non-zero, there would be a one-dimensional $T$-invariant subspace, as $T$ is abelian. This would provide a splitting of the extension $0 \to \G_m \to T \to \Pic \to 0$. We conclude that $F_1|_{C^{\sm}} = F_2|_{C^{\sm}}$ by pulling back along the Abel-Jacobi map. If $\tilde{g}$ denotes the genus of the normalization of $C$, then the dimension of the subspace of pairs of line bundles of rank 1 satisfying $F_1|_{C^{\sm}} = F_2|_{C^{\sm}}$ is $2g - \tilde{g}$. But by Proposition 6 in \cite{Arinkin:2007fk} the strata $\M^{\tilde{g}}$ of curves of geometric genus $\tilde{g}$ has codimension $\geq g - \tilde{g}$. This proves the first part of the claim.

To prove the second assertion it suffices to note that Lemma \ref{lemma:smoothHecke} implies $$\supp \Theta \cap \Pic \times_{\M_g} \Pic \times_{\M_g} \C^{\sm} = \Delta_{\Pic} \times_{\M_g} \C^{sm}. $$
This is sufficient, since an irreducible component of $\supp \Theta$, which does not intersect this smooth locus, must have even higher codimension. Hence we see that every top-dimensional irreducible component intersects $\Delta_{\Picb/\M_g} \times_{\M_g} \C$.\end{proof}
\begin{lemma}\label{lemma:upperbound}
We have $\Theta \in D^{\leq g}(\Pic \times_{\M_g} \Pic \times_{\M_g} \C)$.
\end{lemma}
\begin{proof}
We denote by $\Hb\Pb$ the complex $$R(p \times_{\M_g} \id_{\Picb})_*L(q \times_{\M_g} \id_{\Picb})^*\Pb \in D^b_{coh}(\Picb \times_{\M_g} \Picb \times_{\M_g} \C).$$ We already know from Lemma \ref{lemma:pullCM} that $L(q \times_{\M_g} \id_{\Picb})^*\Pb$ is a Cohen-Macaulay sheaf. As seen in Lemma \ref{lemma:lci} the morphism $\Hh \to \M_g$ is fibrewise irreducible and of dimension $g+1$. In particular we conclude that the dimension of fibres of $\Hh \to \Picb$ is bounded by $g$, hence $\supp H^i(\Hb\Pb)$ is of relative dimension $\leq g-i$ over the parametrizing component $\Picb$.

The integral kernel $\Theta$ is given by convolution $$ \Pb^{\dual} * \Hb\Pb = Rp_{13,*}(Lp_{12}^*\Pb^{\dual} \otimes^L Lp_{23}^*\Hb\Pb). $$ The dimension estimate above implies that $H^i(\Theta) = 0$ if $i > g$.
\end{proof}
\begin{proof}[Proof of Theorem \ref{thm:Hecke}]
We apply Proposition \ref{prop:CMderived} to the integral kernel $\Theta[g]$. We have already checked two of the three necessary conditions in Lemma \ref{lemma:codim} and \ref{lemma:upperbound}. Moreover we know that the theorem is true when restricted to the complement of a codimension two subvariety (cf. Lemma \ref{lemma:smoothHecke}). Therefore it suffices to check the last condition of Proposition \ref{prop:CMderived}: we need to show that $H^i(\mathbb{D}\Theta[g]) = 0$ if $i>g$. From Grothendieck-Serre duality it follows that
$$ \mathbb{D}\Theta = \mathbb{D}Rp_{13,*}(Lp_{12}^*\Pb^{\dual} \otimes^L Lp_{23}^*\Hb\Pb) = (Rp_{13,*}\mathbb{D}(Lp_{12}^*\Pb^{\dual} \otimes^L Lp_{23}^*\Hb\Pb))[g], $$
which in turn can be simplified as
$$ Rp_{13,*}\mathbb{D}(Lp_{12}^*\Pb^{\dual} \otimes^L Lp_{23}^*\Hb\Pb) = Rp_{13,*}[\omega_{\Pic^3}p_{23}^*\omega_{\Pic^2}^{-1}Lp_{12}^*\Pb \otimes^L (Lp_{23}^*\mathbb{D}\Hb\Pb)]. $$
Using Grothendieck-Serre duality again we see that 
$$ \mathbb{D}\Hb\Pb = (p \times \id)_*\mathbb{D}(q \times \id)^*\Pb. $$
According to Lemma \ref{lemma:pullCM} the sheaf $(q \times \id)^*\Pb$ is Cohen-Macaulay, therefore $\mathbb{D}(q \times \id)^*\Pb$ is a sheaf itself. Applying the same reasoning as in Lemma \ref{lemma:upperbound} we see that $\mathbb{D}\Theta \in D^{\leq 0}(\Picb \times_{\M_g} \Picb \times_{\M_g} \C)$. We conclude that $\Theta[g]$ is a Cohen-Macaulay sheaf.
\end{proof}

\bibliographystyle{amsalpha}
\bibliography{master}

\end{document}